\documentclass[10pt]{elsarticle}

\usepackage{amssymb}
\usepackage{amsfonts}
\usepackage{verbatim}
\usepackage{xcolor}
\usepackage{amsmath}
\usepackage{amsthm}
\usepackage{extarrows}

\usepackage[a4paper,left=30mm,width=150mm,top=25mm,bottom=25mm]{geometry}

\usepackage{graphicx}
\DeclareGraphicsExtensions{.eps,.ps,.jpg,.bmp}

\usepackage{booktabs} % for professional tables
\usepackage{adjustbox}
\usepackage{multirow}

\newtheorem{theorem}{Theorem}

\newtheorem{lemma}{Lemma}

\newtheorem{corollary}{Corollary}
\newtheorem{proposition}{Proposition}
\usepackage{lineno,hyperref}
\modulolinenumbers[5]

\usepackage[all]{xy}

\usepackage{pifont}

\journal{Journal}
\biboptions{sort&compress}

\begin{document}

\begin{frontmatter}

\title{Controllability and mixing for acoustic wave motions\tnoteref{mytitlenote}}
\tnotetext[mytitlenote]{This research was partially supported by the National Natural Science Foundation of China (12272297, 12101471 and 11971308) and the Fundamental Research Funds for the Central Universities (104972024KFYjc0065).}
%\tnotetext[mytitlenote]{This research was partially supported by the National Natural Science Foundation of China (12272297).}

\author[mymainaddress]{Zhe Jiao}
\ead{zjiao@nwpu.edu.cn}

%\author[mymainaddress]{Zhe Jiao\corref{mycorrespondingauthor}}
%\ead{zjiao@nwpu.edu.cn}

\author[mymainaddress]{Xiao Li}
\ead{lixiaoo@mail.nwpu.edu.cn}

\author[mythirdaddress]{Qin Zhao\corref{mycorrespondingauthor}}
\cortext[mycorrespondingauthor]{Corresponding author}
\ead{qzhao@whut.edu.cn}

\address[mymainaddress]{School of Mathematics and Statistics, MOE Key Laboratory for Complexity Science in Aerospace, Northwestern Polytechnical University, Xi'an 710129, China}

%\address[mymainaddress]{School of Mathematics and Statistics, Northwestern Polytechnical University, Xi'an 710129, China}

%\address[secondaddress]{MOE Key Laboratory for Complexity Science in Aerospace, Northwestern Polytechnical University, Xi'an 710129, China}

\address[mythirdaddress]{School of Mathematics and Statistics, Wuhan University of Technology, Wuhan 430070, China}

\begin{abstract}
This paper concerns the dynamical behaviors of acoustic wave motion driven by a force acting through the boundary. 
If the boundary force is a suitable control, we show that the dynamical system associated to the acoustic wave motion is exactly controllable. Furthermore, when it is random perturbation of white noise type, we prove that the corresponding stochastic system is strong mixing. The bridge between these two problems is the observability inequality, which will be established in this work. 
\end{abstract}

\begin{keyword}
Wave equation, acoustic boundary condition, observability inequality, exact controllability, mixing
\MSC[2020] 60H15 \sep 35L53 \sep 35R60 \sep 37A25
\end{keyword}

\end{frontmatter}

%\linenumbers

\section{Introduction}
\label{sec:intro}

\subsection{Problem set up and goals}
Let $\Omega\subset \mathbb{R}^3$ be a bounded and simply connected domain that is full of some kind of idealized fluid. The small amplitude acoustic phenomena occurring in $\Omega$ can be modeled below
\begin{equation}
	\rho\mathbf{v}_t = -\nabla p, \quad \kappa p_t = -\nabla\cdot\mathbf{v}, \quad \textrm{in $\mathbb{R}^{+} \times \Omega$}, \label{A1}
\end{equation}
where the vector $\mathbf{v}(t, x)$ is the fluid velocity, $p(t, x)$ the acoustic pressure, $\rho>0$ the uniform density of the fluid and $\kappa>0$ the adiabatic compressibility, see in~\cite{MI}. As the idealized fluid is non-viscous, it is irrotational, which is characterized by having a curl of its velocity field equal to zero, that is, 
\begin{equation} \label{A2}
	\mathrm{curl} \mathbf{v} = 0,  \quad \textrm{in $\mathbb{R}^{+} \times \Omega$}.
\end{equation}

Throughout the paper, the domain $\Omega$ is hypothesized to possess a smooth boundary $\Gamma = \partial\Omega$. Moreover, $\Gamma = \Gamma_0\cup\Gamma_1$, $\overline{\Gamma_0}\cap\overline{\Gamma_1} = \emptyset$ and $\Gamma_1$ is nonempty.
When the acoustic wave is incident on the boundary $\Gamma$, the acoustic pressure $p$ tends to make the boundary move. 
Assume that the portion $\Gamma_0$ of the boundary is rigid and impenetrable to that kind of fluid so that
\begin{equation} 
	 \mathbf{v}\cdot\mathbf{n} = 0,  \quad \textrm{on $\mathbb{R}^{+} \times \Gamma_0$}, \label{B0} 
\end{equation}
where $\mathbf{n}$ is the unit exterior normal vector of the boundary. 
Suppose that the material forming the boundary is an elastic membrane. In this case, the boundary $\Gamma_1$ is referred to as a \emph{non-locally reacting surface}, which can deform under the influence of pressure $p(t, x)$ and exterior force $f(t, x)$. This allows us to use the following mathematical relationship, which is so-called the \emph{acoustic boundary condition}, to describe the motion along the boundary $\Gamma_1$.  
\begin{equation} \label{B1}
\left
    \{
        \begin{array}{ll}
            -p= m v_{tt} - \mathrm{div}_{\Gamma}\left(\sigma\nabla_{\Gamma} v\right) + d v_t + k v + f,\\ 
             \mathbf{n}\cdot \mathbf{v}=-v_t, 
        \end{array} 
\right. \quad \textrm{on $\mathbb{R}^{+} \times \Gamma_1$}.
\end{equation}
Here, $v(t, x)$ is the normal displacement into the domain of a point $x \in \Gamma_1$ at time $t$, $m>0$ is the mass per unit area of $\Gamma_1$, $\sigma\geqslant0$ is the effective stiffness coefficient, $d>0$ is the resistivity, and $k\geqslant0$ is the spring constant. The first equation of wave-type arises from the vibration of the boundary~$\Gamma_1$, while the second equation is derived from the continuity of velocity at the boundary~$\Gamma_1$.

The boundary-value problem~\eqref{A1}-\eqref{B1} is an Eulerian model of the interaction between acoustic wave motions and non-locally reacting surfaces. The physical derivation of this model is given in~\cite[Chapter 7]{mugnolo2024wave}. As a byproduct of the derivation, the following constraint should be added as a constitutive part of this Eulerian model
\begin{equation}\label{Co}
\int_{\Omega}p(t, x)\mathrm{d}x = \frac{1}{\kappa}\int_{\Gamma_1}v(t, x)\mathrm{d}\Gamma, \quad \forall t\geqslant 0,
\end{equation}
which is the Hooke's law written at a macroscopic level.

This constrained Eulerian model~\eqref{A1}-\eqref{Co} can be reformulated to be a potential model. Due to the irrotational property~\eqref{A2}, the velocity field $\mathbf{v}$ can be obtained by taking the gradient of a scalar function $u(t, \mathbf{x})$, that is, $\mathbf{v}=-\nabla u$. Here, $u$ is always called the velocity potential. Consequently, from the first equation of~(\ref{A1}) it implies that $p=\rho u_t$, and by the second equation of~(\ref{A1}) $u$ satisfies the following wave equation
\begin{equation}
	 u_{tt} = c^2\triangle u, \quad \textrm{in $\mathbb{R}^{+} \times \Omega$.} \label{1}
\end{equation}
Here, $c=\sqrt{\frac{1}{\kappa\rho}}$ is the speed of wave propagation. The boundary conditions \eqref{B0} and \eqref{B1} can be rewritten as
\begin{equation} \label{2}
\left
    \{
        \begin{array}{ll}
            -\rho u_t= m v_{tt} - \mathrm{div}_{\Gamma}\left(\sigma\nabla_{\Gamma} v\right) + d v_t + k v + f, \quad \textrm{on $\mathbb{R}^{+} \times \Gamma_1$},\\ 
            \partial_{\mathbf{n}}u =v_t, \quad \textrm{on $\mathbb{R}^{+} \times \Gamma_1$}, \quad \partial_{\mathbf{n}}u = 0, \quad \textrm{on $\mathbb{R}^{+} \times \Gamma_0$}.
        \end{array} 
\right. 
\end{equation}
Plugging $p=\rho u_t$ into the constraint~\eqref{Co} gives the integral condition as follows
\begin{equation}\label{3}
\rho\int_{\Omega}u_t(t, x)\mathrm{d}x = \frac{1}{\kappa}\int_{\Gamma_1}v(t, x)\mathrm{d}\Gamma,  \quad \forall t\geqslant 0.
\end{equation}

In this paper, we will deal with the constrained potential model~\eqref{1}-\eqref{3} supplemented with an initial condition 
\begin{equation}\label{4}
u(0, x)=u_0(x),\quad u_t(0, x)=u_1(x),\quad v(0, x)=v_0(x),\quad v_t(0, x)=v_1(x). 
\end{equation}
Our first aim is to study the following question of exact controllability for the initial-boundary value problem~\eqref{1}-\eqref{4}
\begin{itemize}
\item (Q1) whether there exists a control function $f$ such that for any initial data and pre-specified target data $(u_{T}, v_{T})$, the corresponding solution $(u, v)$ to~\eqref{1}-\eqref{4} satisfying
\[
\left(u(T, x), v(T, x)\right) = (u_{T}, v_{T})
\]
for some $T>0$.
\end{itemize}
The second aim is to investigate the question of mixing for the problem~\eqref{1}-\eqref{4}, specifically,
\begin{itemize}
\item (Q2) whether the law of the solution $(u, v)$ to~\eqref{1}-\eqref{4} converges to an invariant measure (if exists and unique) as time goes to infinity if $f$ is a random perturbation.
\end{itemize}

\subsection{Related work} 
Let us give a short survey of the literature relevant to the study questions (Q1) and (Q2) for the initial-boundary value problem~\eqref{1}-\eqref{4}.

%It can be divided into two categories. 
The physical derivation of acoustic boundary condition ($\sigma = 0$), which is crucial for modeling the interaction between acoustic waves and locally reacting surfaces (cf.~\cite[pp. 259-264]{MI}), was first proposed by Beale and Rosencrans in their seminal work~\cite{BR}. Afterwards, Beale in~\cite{B1,B2} studied the spectral properties of this acoustic wave model posed in both bounded and exterior domains. Based on Beale's work, Littman and Liu~\cite{LL} utilized perturbation theory to study the energy decay and the related
spectral properties of the model in two spatial dimensions with a homogeneous Neumann boundary condition on a portion of the boundary. 
The authors in~\cite{rivera2003polynomial,AN2} obtained polynomial stability for this model via the multiplier method, while~\cite{AN1} investigated strong stability using the semigroup method. Subsequent studies~\cite{Gao2,cavalcanti2020stability,li2025dynamic} provided the well-posedness results and established uniform decay rates by imposing interior localized dampings into the model. Note that in all of them~\cite{rivera2003polynomial,AN2,AN1,Gao2,cavalcanti2020stability,li2025dynamic} the homogeneous Neumann boundary condition is replaced by the homogeneous Dirichlet boundary condition.
Recently, Carri\~ao and Vicente in~\cite{carriao2025stability} studied the acoustic wave model defined in an exterior domain of the disk model of the hyperbolic space, employed the Faedo--Galerkin method to prove the existence of solution, and applied multipliers techniques to study the exponential stability of the model with an interior localized nonlinear damping.

For the case of non-locally reacting surface, Beale in~\cite[Section 6]{B1}  deduced another acoustic boundary condition ($\sigma \neq 0$) on this surface by regarding it as a membrane. However, the complete derivation and comprehensive understanding of this type of boundary condition were provided by Mugnolo and Vitillaro in~\cite[Chapter 7]{mugnolo2024wave}, as mentioned before. Moreover, it should be pointed out that the (constrained) Eulerian model, the Lagrangian model and the (constrained) potential model are derived in~\cite{mugnolo2024wave}. Meanwhile, Vitillaro in~\cite{vitillaro2024three} provided a rigorous study of the relations among these models from a mathematical point of view. In~\cite{mugnolo2024wave}, the authors also proved the well-posedness for the constrained potential model, and studied the long-time behavior of its solutions. In addition, the authors in~\cite{barbieri2022damped, vicente2016uniform,ha2016general,HA2018201,frota2014mixed} have extensively studied the energy decay of potential models with dampings acting both in the domain and on the non-locally reacting surface. It should be noted that the non-locally reacting surface considered in \cite{frota2014mixed} is a two-dimensional bounded manifold with boundary, on which a homogeneous Dirichlet boundary condition is imposed. 

When the locally reacting surface is subject to an extra force ($f\neq 0$), we have investigated questions (Q1) and (Q2) in~\cite{jiao2025mixing}. As a follow-up to that work, the present paper will focus on questions (Q1) and (Q2) for the initial-boundary value problem~\eqref{1}-\eqref{4}. To date, there has been limited work (cf. \cite{avalos2003exact}) on these questions. It is worth noting that only question~(Q1) was addressed in \cite{avalos2003exact}, where the model differs from the one considered herein. A detailed comparison between this paper and related studies in the literature is summarized in Table~\ref{overview}.
\begin{table}[htbp]
\begin{adjustbox}{center}
\scalebox{0.8}{
\begin{tabular}{ccccccc}
\toprule
 Paper& $\Gamma_1$  & $\Gamma_0$ & interior damping & well-posedness & stability & controllability \\
\toprule
\cite{BR,B1,B2}&  $\sigma = 0$, $d\geqslant0$, $k\geqslant 0$, $f = 0$ & $\emptyset$ & /   & \ding{51}  & / & / \\  \cmidrule{2-2}
\cite{LL}& \multirow{5}{*}{$\sigma = 0$, $d>0$, $k\geqslant 0$, $f = 0$}  & NBC & /   &  / &  \ding{51} & / \\  \cmidrule{3-7}
\cite{rivera2003polynomial,AN2,AN1} & & DBC & /   & \ding{51}  &  \ding{51}& / \\
\cite{Gao2}&    & DBC & localized, linear  &  \ding{51} &\ding{51} & / \\
\cite{cavalcanti2020stability,li2025dynamic}&  & DBC & localized, nonlinear   &  \ding{51} &  \ding{51}& /  \\
\cite{carriao2025stability}$^{\dag}$ &    &  $\emptyset$ & localized, nonlinear &  \ding{51} &  \ding{51}& / \\
\midrule
\cite{B1}$^{\natural}$ & {$\sigma \neq 0$, $d\geqslant0$, $k\geqslant 0$, $f = 0$} & NBC & / & \ding{51} & / & / \\  \cmidrule{2-2}
\cite{mugnolo2024wave,vitillaro2024three}& & NBC & / & \ding{51} & \ding{51} & /  \\
\cite{barbieri2022damped}&  \multirow{5}{*}{$\sigma \neq 0$, $d>0$, $k\geqslant 0$, $f = 0$}  & NBC & linear  & \ding{51}  &  \ding{51} & / \\ \cmidrule{3-7}
\cite{vicente2016uniform}&  & DBC & nonlinear & / &\ding{51}  & / \\
\cite{ha2016general,HA2018201}&  & DBC & nonlinear & / & \ding{51} & / \\
\cite{frota2014mixed}$^{\natural}$ &  & DBC & nonlinear & \ding{51} & \ding{51} & / \\
\midrule
\cite{jiao2025mixing} & $\sigma = 0$, $d>0$, $k\geqslant 0$, $f \neq 0$ & DBC, NBC  &/ & \ding{51} & / & \ding{51} \\
\midrule
 \cite{avalos2003exact}$^{\sharp}$ & $\sigma \neq 0$, $d=k=0$, $f \neq 0$ & {NBC}  &/ & / & / & \ding{51} \\
This work & $\sigma \neq 0$, $d>0$, $k\geqslant 0$, $f \neq 0$& NBC & / & \ding{51} & / & \ding{51} \\
\bottomrule
\end{tabular}
}
\end{adjustbox}
\caption{Overview of the investigation about the wave equation with acoustic boundary conditions. DBC: Dirichlet boundary condition. NBC: Neumann boundary condition. The superscript $^{\dag}$ means that the problem studied in the paper is posed in an exterior domain of the disk model of the hyperbolic space; $^{\natural}$ means Dirichlet boundary condition on $\partial\Gamma_1$; $^{\sharp}$ means Neumann boundary condition on $\partial\Gamma_1$.}
\label{overview}
\end{table}

\subsection{Main contributions} 

Our main contribution is twofold. 
\begin{itemize}
\item Using Lions' Hilbert Uniqueness Method (HUM), we establish an equivalence between exact controllability and observability. Thanks to Lemma~\ref{regularity} and Doob's theorem (cf.~\cite[Theorem 5.3.1]{PratoZabczyk}), the mixing property follows from exact controllability through the $t$-regularity property. Consequently, questions (Q1) and (Q2) are answered once the observability inequality is verified, as shown in Theorem~\ref{observability}. The relationship among these four properties can be illustrated in the following diagram.
%\begin{equation*}
%\xymatrix{
%\textrm{Controllability}  \ar@{->}[d]  \ar@{->}[r]  & \textrm{Observability}\ar@{->}[d]\\
%\textrm{$t$-regularity} \ar@{->}[r] &\textrm{Mixing}
%}
%\end{equation*}
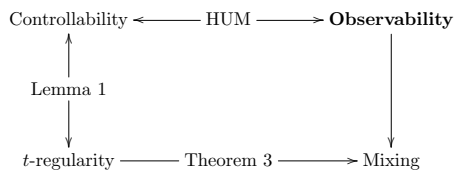
\begin{equation*}
\scalebox{0.7}{
\xymatrix{
\textrm{Controllability}  & \textrm{HUM}\ar@{->}[r] \ar@{->}[l]& \textbf{Observability}\ar@{->}[2,0]\\
 \textrm{Lemma~\ref{regularity}}\ar@{->}[d] \ar@{->}[u] &  & \\
\textrm{$t$-regularity} \ar@{-}[r] & \textrm{Theorem~\ref{mixing}}\ar@{->}[r]&\textrm{Mixing}
}
}
\end{equation*}
\item As a secondary outcome of our main results, we derive an explicit form of control function that guarantees the property of exact controllability. Moreover, in the special case where $\sigma=0$, the problem~\eqref{1}-\eqref{4} considered in this paper reduces to that in~\cite{jiao2025mixing} so that our main results improve those in~\cite{jiao2025mixing}. 
\end{itemize}

\subsection{Organization of the paper} 

This remaining paper is organized as follows. In Section~\ref{sec:preliminary}, we provide another derivation of wave equation with acoustic boundary conditions, and recall some classical results of control theory for abstract linear system. Subsequently, we give a concise answer to problem (Q1) in Section~\ref{sec:ctrl_obs}. With the aid of the observability inequality proposed in Section~\ref{sec:ctrl_obs}, we address problem (Q2) in Section~\ref{sec:mixing}. Finally, some conclusions are presented in Section~\ref{sec:conclusion}.

\subsection*{Notation}  

Let $\mathcal{H}$ and $\mathcal{U}$ be two Hilbert spaces. 
We shall use the following notations: 
%$B(\mathcal{H})$ denotes the Borel $\sigma$-algebra over $\mathcal{H}$; $\mathcal{P}(\mathcal{H})$ is the set of probability Borel measure on $(\mathcal{H}, B(\mathcal{H}))$; $C_b(\mathcal{H})$ denotes the space of bounded continuous functions on $\mathcal{H}$.
\begin{itemize}
\item $B(\mathcal{H})$ denotes the Borel $\sigma$-algebra over $\mathcal{H}$. 
\item $\mathcal{P}(\mathcal{H})$ is the set of probability Borel measure on $(\mathcal{H}, B(\mathcal{H}))$. 
\item $C_b(\mathcal{H})$ denotes the space of bounded continuous functions on $\mathcal{H}$. 
\end{itemize}
In addition, let $\mathrm{Dom}(\cdot)$ denote the domain of an operator, $\mathrm{Ran}(\cdot)$ denote the range of an operator, and $C(\cdot)$ denote positive numbers, which may depend on the quantities mentioned in the brackets.

\section{Preliminary}
\label{sec:preliminary} 

In this section, we shall give another method to deduce the model~\eqref{1}-\eqref{3} of acoustic motion. Our approach of this derivation above is inspired by the work of~\cite{goldstein2006}. Afterwards, we recall, for the reader's convenience, the classical results of controllability and observability for abstract linear control system (see~\cite[Chapter 2.3-2.4, pp. 51-61]{Coron} and references therein).

\subsection{Derivation via the principle of stationary action}
\label{subsec:derivation}

To derive the equations~\eqref{1}-\eqref{3} from the principle of stationary action, the action functional of the acoustic motion can be modeled by 
\begin{equation*}
\mathbf{S}(u, v) :=\int_{0}^{T}\left[\mathbf{K}(u_t, v_t) - \mathbf{V}(u, v, u_t)\right] \mathrm{d}t, 
\end{equation*}
where
\begin{equation*}
\begin{split}
 \mathbf{K}(u_t, v_t)&:= \frac{1}{2}\int_{\Omega}c^{-2}u_t^2\mathrm{ d}x + \frac{1}{2}\int_{\Gamma_1}\frac{m}{\rho}v_t^2 \mathrm{d}\Gamma, \\
\mathbf{V}(u, v, u_t)&:=\frac{1}{2}\int_{\Omega}|\nabla u|^2 \mathrm{d}x + \frac{1}{2}\int_{\Gamma_1}\frac{\sigma}{\rho}|\nabla_{\Gamma_1}v|^2 \mathrm{d}\Gamma + \frac{1}{2}\int_{\Gamma_1}\frac{k}{\rho}v^2\mathrm{ d}\Gamma + \int_{\Gamma_1}vu_t \mathrm{d}\Gamma.
\end{split}
\end{equation*}

Now, we apply the calculus of variations to find equations of motion for $u$ and $v$ in domain $\overline{\Omega}$. Let $\omega: [0, T]\times\Omega\rightarrow\mathbb{R}$ and $\eta: [0, T]\times\Gamma\rightarrow\mathbb{R}$ with $\omega(0)=\omega(T)=0$ and $\eta(0)=\eta(T)=0$ be variations of $u$ and $v$. Then integrating by parts, one can compute
%\begin{equation}
%\begin{split}
%&\mathbf{H}(u+\epsilon\omega, v+\epsilon\eta) \\
%&=  \frac{1}{2}\int_{0}^{T}\int_{\Omega}\left[c^{-2}(u+\epsilon\omega)_t^2 - |\nabla(u+\epsilon\omega)|^2\right]dxdt\\
%&\quad+\frac{1}{2}\int_{0}^{T}\int_{\Gamma_1}\left[ \frac{m}{\rho}(v+\epsilon\eta)_t^2 -\frac{\sigma}{\rho}|\nabla_{\Gamma}(v+\epsilon\eta)|^2 - \frac{k}{\rho}(v+\epsilon\eta)^2 - 2(v+\epsilon\eta)(u+\epsilon\omega)_t\right]dSdt
%\end{split}
%\end{equation}
%\begin{equation}
%\begin{split}
%&\mathbf{H}(u+\epsilon\omega, v+\epsilon\eta) - \mathbf{H}(u, v)\\
%&= \epsilon\int_{0}^{T}\int_{\Omega}\left[ c^{-2}u_t\omega_t - \nabla u\cdot\nabla\omega\right]dxdt\\
%&\quad+\epsilon\int_{0}^{T}\int_{\Gamma_1}\left[ \frac{m}{\rho} v_t\eta_t -\frac{\sigma}{\rho}\nabla_{\Gamma}v\cdot\nabla_{\Gamma}\eta - \frac{k}{\rho} v\eta - \eta u_t -  v\omega_t \right]dSdt +\mathcal{O}(\epsilon^2)\\
%&= \epsilon\int_{\Omega} \left(c^{-2}[u_t\omega]^{T}_{t=0} + \int_{0}^{T}\left[ -c^{-2}u_{tt} + \triangle u\right]\omega dt\right)dx - \epsilon\int_{0}^{T}\int_{\Gamma}\partial_{\mathbf{ n}}u w dSdt\\
%&\quad+\epsilon\int_{\Gamma_1}\left([\frac{m}{\rho} v_t\eta + v w]^{T}_{t=0} + \int_{0}^{T}\left[ -\frac{m}{\rho} v_{tt}\eta + \frac{1}{\rho}\mathrm{div}(\sigma\nabla_{\Gamma}v)\eta - \frac{k}{\rho} v\eta - \eta u_t +  v_t\omega \right]dt\right)dS +\mathcal{O}(\epsilon^2)\\
%\end{split}
%\end{equation}
\begin{equation*}
\begin{split}
&\mathbf{S}(u+\epsilon\omega, v+\epsilon\eta) - \mathbf{S}(u, v)\\
&= \epsilon\int_{0}^{T}\int_{\Omega}\left[ c^{-2}u_t\omega_t - \nabla u\cdot\nabla\omega\right]\mathrm{d}x\mathrm{d}t\\
&\quad+\epsilon\int_{0}^{T}\int_{\Gamma_1}\left[ \frac{m}{\rho} v_t\eta_t -\frac{\sigma}{\rho}\nabla_{\Gamma}v\cdot\nabla_{\Gamma}\eta - \frac{k}{\rho} v\eta - \eta u_t -  v\omega_t \right]\mathrm{d}\Gamma \mathrm{d}t +\mathcal{O}(\epsilon^2)
\end{split}
\end{equation*}
and deduce
\begin{equation}\label{calculus}
\begin{split}
&\mathbf{S}(u+\epsilon\omega, v+\epsilon\eta) - \mathbf{S}(u, v)\\
&= \epsilon\int_{\Omega} \left(c^{-2}[u_t\omega]^{T}_{t=0} + \int_{0}^{T}\left[ -c^{-2}u_{tt} + \triangle u\right]\omega dt\right)\mathrm{d}x\\
&\quad - \epsilon\int_{0}^{T}\int_{\Gamma_0}\partial_{\mathbf{ n}}u w d\Gamma dt - \epsilon\int_{0}^{T}\int_{\Gamma_1}\left(\partial_{\mathbf{ n}}u - v_t\right)\omega \mathrm{d}\Gamma \mathrm{d}t\\
&\quad+\epsilon\int_{\Gamma_1}\left([\frac{m}{\rho} v_t\eta + v w]^{T}_{t=0} + \int_{0}^{T}\left[ -\frac{m}{\rho} v_{tt}\eta + \frac{1}{\rho}\mathrm{div}(\sigma\nabla_{\Gamma}v)\eta - \frac{k}{\rho} v\eta - \eta u_t  \right]\mathrm{d}t\right)\mathrm{d}\Gamma +\mathcal{O}(\epsilon^2)\\
&= \epsilon\int_{\Omega}\int_{0}^{T}\left[ -c^{-2}u_{tt} + \triangle u\right]\omega dtdx - \epsilon\int_{0}^{T}\int_{\Gamma_0}\partial_{\mathbf{ n}}u w \mathrm{d}\Gamma \mathrm{d}t - \epsilon\int_{0}^{T}\int_{\Gamma_1}\left(\partial_{\mathbf{ n}}u - v_t\right)\omega \mathrm{d}\Gamma \mathrm{d}t\\
&\quad+\epsilon\int_{\Gamma_1}\int_{0}^{T}\left[ -\frac{m}{\rho} v_{tt}\eta + \frac{1}{\rho}\mathrm{div}(\sigma\nabla_{\Gamma}v)\eta - \frac{k}{\rho} v\eta - \eta u_t  \right]\mathrm{d}t\mathrm{d}\Gamma +\mathcal{O}(\epsilon^2).
\end{split}
\end{equation}
The principle of stationary action states that the transition of the acoustic motion from an initial state at $t=0$ to a final state $t=T$ is the one for which the action functional $\mathbf{S}(u, v)$ is stationary to first order, which implies all terms in $\mathcal{O}(\epsilon)$ vanish for arbitrary $\omega$ and $\eta$. Therefore, equation~\eqref{calculus} leads to the potential model~\eqref{1}-\eqref{2}.

To derive the constraint~\eqref{3}, we should notice that
\begin{equation*} 
\begin{split}
	\frac{\mathrm{d}}{\mathrm{d}t}\left(\rho\int_{\Omega}u_t \mathrm{d}x - \frac{1}{\kappa}\int_{\Gamma_1}v \mathrm{d}\Gamma \right) &= \rho\int_{\Omega}u_{tt} \mathrm{d}x - \frac{1}{\kappa}\int_{\Gamma_1}v_t \mathrm{d}\Gamma\\
	&= \frac{1}{\kappa}\left(\int_{\Omega}\triangle u \mathrm{d}x - \int_{\Gamma_1}\partial_{\mathbf{n}}u \mathrm{d}\Gamma\right) =0,
\end{split}
\end{equation*}
which implies
\begin{equation}\label{Hooke} 
	\rho\int_{\Omega}u_t \mathrm{d}x - \frac{1}{\kappa}\int_{\Gamma_1}v \mathrm{d}\Gamma = \rho\int_{\Omega}u_1 \mathrm{d}x - \frac{1}{\kappa}\int_{\Gamma_1}v_0 \mathrm{d}\Gamma.
\end{equation}
This concludes that the velocity potential $u$ and the boundary displacement $v$ satisfy the constraint given by equation~\eqref{3}, provided that the initial data conform to the Hooke's law.

\subsection{Abstract linear system}
\label{subsec:obs}

Let $\mathcal{H}$ be the state space and $\mathcal{U}$ be the control space.
The control system we consider here is
\begin{equation} 
\left
    \{
        \begin{array}{ll}
            \mathbf{Y}^{\prime}(t)=\mathcal{A}\mathbf{Y}(t)+ \mathcal{B}\mathbf{u}(t), \quad t\in[0, T],\\
            \mathbf{Y}(0)=\mathbf{y}_0 \in \mathcal{H},  
        \end{array} \label{eq:controlled}
\right.
\end{equation}
where the state is $\mathbf{ Y}(t)\in \mathcal{H}$ and the control is $\mathbf{u}(t)\in\mathcal{U}$. In addition, the following assumptions are imposed on this system.
\begin{itemize}
\item $\mathcal{A}$ is the infinitesimal generator of a strongly continuous semigroup $S(t)$, $t\geqslant 0$, on $\mathcal{H}$.
\item $\mathcal{B}$ is a linear continuous mapping from $\mathcal{U}$ into $\mathcal{H}$.
%\item The linear operator $Q_t$, $t > 0$, defined by
%\[
%	Q_{t}\mathbf{x}=\int_{0}^{t}S(s)BB^{\ast}S^{\ast}(s)\mathbf{x} ds, \quad \mathbf{x}\in\mathcal{H}
%\]
%is of trace class.
\end{itemize}
As usual, we call the system~\eqref{eq:controlled} is \emph{exactly controllable} in time $T$ if, for any $\mathbf{y}_0$, $\mathbf{y}_1\in \mathcal{H}$, there exists a control function $\mathbf{u}\in \mathcal{U}$ such that the solution of the system~\eqref{eq:controlled} satisfies $\mathbf{Y} (T) = \mathbf{y}_1$. 

One easily sees that the exact controllability of system~\eqref{eq:controlled} implies it is \emph{null controllable} in time $T$, which means given any $\mathbf{y}_0\in \mathcal{H}$, there exists a control function $\mathbf{u}\in \mathcal{U}$ such that $\mathbf{Y} (T)=0$. 
The converse is true due to the fact that $S(t)$ is strongly continuous. 
Indeed, for every $\mathbf{y}_0\in \mathcal{H}$ and for every $\mathbf{y}_1\in \mathcal{H}$, by employing the null controllability of system~\eqref{eq:controlled} to the initial data $\mathbf{y}_0 - S(-T)\mathbf{y}_1$ we know there exists a control $\widetilde{\mathbf{u}}\in \mathcal{U}$ such that the solution of the following system
\begin{equation*} 
{\widetilde{\mathbf{Y}}}^{\prime}(t)=A(\widetilde{\mathbf{Y}}(t))+ B\widetilde{\mathbf{u}}(t), \quad \widetilde{\mathbf{Y}}(0)=\mathbf{y}_0 - S(-T)\mathbf{y}_1
\end{equation*}
satisfies $\widetilde{\mathbf{Y}}(T) = 0$. Assume that $\mathbf{Y}(t)$ is the solution of the problem
\begin{equation*} 
{\mathbf{Y}}^{\prime}(t)=A(\mathbf{Y}(t))+ B\widetilde{\mathbf{u}}(t), \quad \mathbf{Y}(0)=\mathbf{y}_0. 
\end{equation*}
Thus, we obtain $ \mathbf{Y}(t) - \widetilde{\mathbf{Y}}(t) = S(t)S(-T)\mathbf{y}_1 = S(t-T)\mathbf{y}_1$, which implies $ \mathbf{Y}(T) = \mathbf{y}_1$. This concludes that the system~\eqref{eq:controlled} possesses the property of exact controllability in time $T$.

Subsequently, we shall provide a sufficient and necessary condition for the exact controllability of the system~\eqref{eq:controlled}. Let us define a linear map
\[
\mathcal{F}_{T}: \mathcal{U}\rightarrow \mathcal{H}, \quad \mathcal{F}_{T}\mathbf{ u} = \check{\mathbf{Y}}(T),
\]
where $\check{ \mathbf{Y}}(t)\in C([0, T]; \mathcal{H})$ is the solution of the system~\eqref{eq:controlled} with $\mathbf{y}_0=0$.
Note that Ran$F_T$ consists of all states reachable in time $T$ from the zero state. 
Assume that $\mathcal{F}_{T}$ is onto. Let $\mathbf{y}_0$, $\mathbf{y}_1\in {H}$ and $\widehat{\mathbf{Y}}(t)$ be the solution of the problem
\begin{equation*} 
\widehat{\mathbf{Y}}^{\prime}(t)=A(\widehat{\mathbf{Y}}(t)), \quad \widehat{\mathbf{Y}}(0)=\mathbf{y}_0. 
\end{equation*}
By the onto property of $\mathcal{F}_{T}$, there exists a function $\widehat{\mathbf{u}}\in \mathcal{U}$ such that $\mathcal{F}_{T}\widehat{\mathbf{u}} = \check{\mathbf{Y}}(T) = \mathbf{y}_1 - \widehat{\mathbf{Y}}(T)$. Suppose that $\mathbf{Y}(t)$ is the solution of the problem
\begin{equation*} 
{\mathbf{Y}}^{\prime}(t)=A(\mathbf{Y}(t))+ B\widehat{\mathbf{u}}(t), \quad \mathbf{Y}(0)=\mathbf{y}_0. 
\end{equation*}
Then we deduce $\mathbf{Y}(t) = \widehat{\mathbf{Y}}(t) + \check{ \mathbf{Y}}(t)$ and $\mathbf{Y}(T) = \mathbf{y}_1$,
which proves that the system~\eqref{eq:controlled} is exactly controllable in time $T$. Obviously, due to the definition of $\mathcal{F}_{T}$, the property of exact controllability yields $\mathcal{F}_{T}$ is onto. In conclusion, it follows that the system~\eqref{eq:controlled} is exactly controllable in time $T$ if and only if $\mathcal{F}_{T}$ is onto. 

Let us recall a classical result of functional analysis (see \cite[Theorem 4.15, pp. 97]{rudin1991functional}) that $\mathcal{F}_{T}$ is onto if and only if there exists $c>0$ such that
\begin{equation} \label{ineq:obs}
\|\mathcal{F}^{\ast}_{T}(z)\|_{\mathcal{U}} \geqslant c\|z\|_{\mathcal{H}}, \quad \forall z\in \mathcal{H}.
\end{equation}
Inequality \eqref{ineq:obs} is usually called \emph{observability inequalities} for the linear control system~\eqref{eq:controlled}. Therefore, in order to establish the exact controllability of the system as~\eqref{eq:controlled}, one only needs to prove the corresponding observability inequality.

\section{Controllability and Observability}
\label{sec:ctrl_obs}
In the following, we shall apply the general framework presented in Section~\ref{subsec:obs} to answer (Q1) for problem~\eqref{1}-\eqref{4}. There are two key points that need to be addressed: first, selecting appropriate state and control spaces; second, proving the observability inequality.

\subsection{Abstract evolution equation}
\label{subsec:evolution}

The problem~\eqref{1}-\eqref{4} will be written as an abstract evolution equation. To do so, we now introduce the phase space
\[
\mathcal{H}_0 = H^1(\Omega)\times H^1(\Gamma_1) \times L^2(\Omega) \times L^2(\Gamma_1).
\]
Let $H^1_{c}(\Omega)$ denote the quotient space of $H^1(\Omega)$ modulo constants, which is isomorphic to $\dot{H}^1(\Omega)$ given by
\[
\dot{H}^1(\Omega) := \left\{\phi(x)\in H^1(\Omega): \int_{\Omega}\phi(x)\mathrm{d}x=0\right\}. 
\]
Then $H^1(\Omega)$ has the standard orthogonal splitting $H^1(\Omega) = \dot{H}^1(\Omega)\bigoplus \mathbb{C}$.
Accordingly, our phase space has the corresponding orthogonal splitting 
\[
\mathcal{H}_0 = \mathbb{C}_1\bigoplus \mathcal{H}_1,
\] 
where $\mathbb{C}_1$ and $\dot{\mathcal{H}}$ are given by
\[
	\mathbb{C}_1= \mathbb{C}1_{\mathcal{H}}, \quad 1_{\mathcal{H}} = (1, 0, 0, 0), \quad \mathcal{H}_1 = \dot{H}^1(\Omega)\times H^1(\Gamma_1) \times L^2(\Omega) \times L^2(\Gamma_1).
\]
To deal with the constraint~\eqref{3}, we shall also set
\[
	{\mathcal{H}_2} :=\left\{\mathbf{ f}=(f_1, f_2, f_3, f_4)\in\mathcal{H}_1: \rho\int_{\Omega}f_3 \mathrm{d}x = \frac{1}{\kappa}\int_{\Gamma_1}f_2 \mathrm{d}\Gamma\right\}.
\]

Let $\mathcal{H} := {\mathcal{H}_1}\bigcap {\mathcal{H}_2}$.
In this case, we equip ${\mathcal{H}}$ with the following inner product
\[
	(\mathbf{ f}, \mathbf{ g})_{\mathcal{H}}=\int_{\Omega}\left[c^{-2}f_3{\overline{g_{3}}} +\nabla f_1\cdot{\nabla \overline{g_1}}\right]\mathrm{d}x+ \int_{\Gamma_{1}}\left[\frac{m}{\rho}f_4{\overline{g_{4}}} + \frac{\sigma}{\rho}\nabla_{\Gamma} f_2\cdot{\nabla_{\Gamma} \overline{g_2}} + \frac{k}{\rho}f_2{\overline{g_{2}}}\right]\mathrm{d}\Gamma,
\]
where $\mathbf{ f} =(f_1, f_2, f_3, f_4)\in\mathcal{H}$, and $\mathbf{ g } = (g_1, g_2, g_3, g_4)\in\mathcal{H}$.
This inner product induces a norm, denoted by $\|\cdot\|_{\mathcal{H}}$. Here, we highlight that the strategy for constructing an appropriate space for our problems is based on the fascinating work presented in~\cite[Chapter 6]{mugnolo2024wave}. 
%This inner product induces a norm, denoted by $\|\cdot\|_{\mathcal{H}}$, not just a seminorm. Here, we highlight that the strategy for constructing an appropriate space for our problems is based on the fascinating work presented in~\cite[Chapter 6]{mugnolo2024wave}. 

We now define an operator $\mathcal{A}: \mathrm{Dom}(\mathcal{A})\subset \mathcal{H}\rightarrow \mathcal{H}$ by
%\begin{equation*}       
%\mathcal{A}=
%\left(                 
%  \begin{array}{cccc}   
%    0 & 1 & 0 & 0\\  
%    c^2\Delta & 0 & 0 & 0\\  
%     0 & 0 & 0 & 1\\
%     0 & -\frac{\rho}{m} & -\frac{k}{m}+\frac{1}{m}\mathrm{div}(\sigma\nabla_{\Gamma}\cdot)  & -\frac{d}{m} \\
%  \end{array}
%\right),                        
%\end{equation*}
\begin{equation*}       
\mathcal{A}\mathbf{ f}=
\left(                 
  \begin{array}{c}   
    f_3\\  
    f_4\\  
    c^2\Delta f_1 \\
     \frac{1}{m}\left[\mathrm{div}(\sigma\nabla_{\Gamma}f_2) -d f_4-k f_2 -\rho f_3|_{\Gamma_1}\right] \\
  \end{array}
\right)                       
\end{equation*}
with the domain of the operator $\mathcal{A}$ 
\begin{equation*} 
\begin{split}
	\mathrm{Dom}(\mathcal{A})=\Big\{\mathbf{f}\in [\dot{H}^1(\Omega)\times H^1(\Gamma_1)]^2: & \quad \Delta f_1\in L^2(\Omega), \quad \partial_{\mathbf{n}} f_1|_{\Gamma_0}=0, \quad \partial_{\mathbf{n}} f_1|_{\Gamma_1}=f_4,\\
	&\quad  \rho\int_{\Omega}f_3 \mathrm{d}x = \frac{1}{\kappa}\int_{\Gamma_1}f_2 \mathrm{d}\Gamma, \quad \mathrm{div}(\sigma\nabla_{\Gamma}f_2)\in L^2(\Gamma_1)\Big\}.
\end{split}
\end{equation*}
Here, $ f_3|_{\Gamma_1}$ in the last component of $\mathcal{A}\mathbf{ f}$ should be understood as the trace of $f_3$ on $\Gamma_1$, and the condition $\partial_{\mathbf{n}} f_1|_{\Gamma_0}=0$ and $\partial_{\mathbf{n}} f_1|_{\Gamma_1}=f_4$ are interpreted in the weak sense
\begin{eqnarray*}
	\int_{\Omega}\left[\Delta f_1 {v} +\nabla f_1\cdot {\nabla v}\right]\mathrm{d}x =\int_{\Gamma_1}f_4{v} \mathrm{d}\Gamma,  \quad \forall v\in H^1(\Omega).
\end{eqnarray*}
From $\rho\int_{\Omega}c^2\Delta f_1\mathrm{d}x = \frac{1}{\kappa}\int_{\Gamma_1}f_4 \mathrm{d}\Gamma$, we can deduce $\mathcal{A}\mathbf{ f}\in \mathcal{H}$. Then the construction of $\mathcal{A}$ is reasonable and necessary for further analysis.

By setting $\mathbf{X}(t)=\left(u, v, u_t, v_t\right)^\top$ and $\mathbf{x} = (u_0, v_0, u_1, v_1)^\top$, the initial-boundary value problem~\eqref{1}-\eqref{4} can be formally reformulated to be the following evolution equation in $\mathcal{H}$
\begin{equation} \label{evolution}
\left
    \{
        \begin{array}{ll}
            \mathbf{X}^{\prime}=\mathcal{A}\mathbf{X} + \mathcal{B}\mathbf{ u},\\
            \mathbf{X}(0)=\mathbf{x},    
        \end{array}
\right.
\end{equation}
where $\mathcal{B} = \mathrm{diag}\left\{0, 0, 0, \frac{1}{m} \right\}$ and $\mathbf{ u} = f\in \mathcal{U}$.
Subsequently, an important property of the operator $\mathcal{A}$ are provided by the following lemma.
\begin{proposition}\label{prop:operator}
$\mathcal{A}$ is a densely defined closed linear operator, which is the infinitesimal generator of a $C_0$-semigroup of contractions on $\mathcal{H}$, denoted by ${S}(t)$, $t\geqslant 0$.
\end{proposition}
\begin{proof}
Note that $\mathrm{Dom}(\mathcal{A})$ is dense in $\mathcal{H}$. Then $\mathcal{A}$ is densely defined. Moreover, we can uniquely define the adjoint $\mathcal{A}^{\ast}$ of $\mathcal{A}$ by
\begin{equation} \label{duality}     
\mathcal{A}^{\ast}\mathbf{ f}=
\left(                 
  \begin{array}{c}   
    -f_3\\  
    -f_4\\  
    -c^2\Delta f_1 \\
     -\frac{1}{m}\left[\mathrm{div}(\sigma\nabla_{\Gamma}f_2) + d f_4-k f_2 -\rho f_3|_{\Gamma_1}\right] \\
  \end{array}
\right),                        
\end{equation}
whose domain is given by
\begin{equation*} 
\begin{split}
	\mathrm{Dom}(\mathcal{A}^{\ast})=\Big\{\mathbf{g}\in [\dot{H}^1(\Omega)\times H^1(\Gamma_1)]^2: & \quad \Delta g_1\in L^2(\Omega), \quad \partial_{\mathbf{n}} g_1|_{\Gamma_0}=0, \quad \partial_{\mathbf{n}} g_1|_{\Gamma_1}=g_4,\\
	&\quad  \rho\int_{\Omega}g_3 \mathrm{d}x = \frac{1}{\kappa}\int_{\Gamma_1}g_2 \mathrm{d}\Gamma, \quad \mathrm{div}(\sigma\nabla_{\Gamma}g_2)\in L^2(\Gamma_1)\Big\}.
\end{split}
\end{equation*}

Due to the calculation
\begin{equation*}  
\begin{split}
(\mathcal{A}\mathbf{ f}, \mathbf{f})_{\mathcal{H}} =& \int_{\Omega}\left[c^{-2}c^2\Delta f_1{\overline{f_{3}}} +\nabla f_3\cdot{\nabla \overline{f_1}}\right]\mathrm{d}x+ \int_{\Gamma_{1}}\left[\frac{\sigma}{\rho}\nabla_{\Gamma} f_4\cdot{\nabla_{\Gamma} \overline{f_2}} + \frac{k}{\rho}f_4{\overline{f_{2}}}\right]\mathrm{d}\Gamma\\
&+\int_{\Gamma_{1}}\frac{m}{\rho}  \frac{1}{m}\left[\mathrm{div}(\sigma\nabla_{\Gamma}f_2) -d f_4-k f_2 -\rho f_3\right]  {\overline{f_{4}}} \mathrm{d}\Gamma\\
%=& \int_{\Gamma}\partial_{\mathbf{n}} f_1{\overline{f_{3}}}d\Gamma - \int_{\Gamma_{1}}f_3 {\overline{f_{4}}} d\Gamma + \int_{\Gamma_{1}}\frac{k}{\rho}\left[f_4{\overline{f_{2}}} - f_{2}\overline{ f_4} \right]d\Gamma - \int_{\Gamma_{1}}\frac{d}{\rho}f_4 {\overline{f_{4}}} d\Gamma\\
%&+\int_{\Omega}\left[-\nabla f_1{\overline{f_{3}}} +\nabla f_3\cdot{\nabla \overline{f_1}}\right]dx+\int_{\Gamma_{1}}\frac{\sigma}{\rho}\left[\nabla_{\Gamma} f_4\cdot{\nabla_{\Gamma} \overline{f_2}} - {\nabla_{\Gamma}{f_2}} \cdot \nabla_{\Gamma} \overline{f_4}\right]d\Gamma\\
=& \int_{\Gamma_1} f_4{\overline{f_{3}}}d\Gamma - \int_{\Gamma_{1}}f_3 {\overline{f_{4}}} \mathrm{d}\Gamma + \int_{\Gamma_{1}}\frac{k}{\rho}\left[f_4{\overline{f_{2}}} - f_{2}\overline{ f_4} \right]\mathrm{d}\Gamma - \int_{\Gamma_{1}}\frac{d}{\rho}f_4 {\overline{f_{4}}} \mathrm{d}\Gamma\\
&+\int_{\Omega}\left[-\nabla f_1{\overline{f_{3}}} +\nabla f_3\cdot{\nabla \overline{f_1}}\right]\mathrm{d}x+\int_{\Gamma_{1}}\frac{\sigma}{\rho}\left[\nabla_{\Gamma} f_4\cdot{\nabla_{\Gamma} \overline{f_2}} - {\nabla_{\Gamma}{f_2}} \cdot \nabla_{\Gamma} \overline{f_4}\right]\mathrm{d}\Gamma
\end{split}
\end{equation*}
for any $\mathbf{ f}\in\mathrm{Dom}(\mathcal{A})$, we get 
\[
\Re(\mathcal{A}\mathbf{ f}, \mathbf{f})_{\mathcal{H}} = - \int_{\Gamma_{1}}\frac{d}{\rho}|f_4 |^2 \mathrm{d}\Gamma \leqslant 0,
\]
which means $\mathcal{A}$ is dissipative. 

To show $\mathcal{A}$ is closable, we use the method of contradiction. Suppose it is not true. Then there is a sequence $\mathbf{f}_n\in\mathrm{Dom}(\mathcal{A})$ such that $\mathbf{f}_n\rightarrow 0$ and $\mathcal{A}\mathbf{f}_n\rightarrow \mathbf{h}$ with $\|\mathbf{h}\|_\mathcal{H}=1$. From the fact that $\mathcal{A}$ is dissipative, it follows that for every $\lambda>0$ and $\mathbf{f}\in \mathrm{Dom}(\mathcal{A})$
 \begin{eqnarray*}
\begin{aligned}
	\|(\frac{1}{\lambda}-\mathcal{A})(\mathbf{f}+\lambda^{-1}\mathbf{f}_n)\|\geq \frac{1}{\lambda}\|\mathbf{f}+\lambda^{-1}\mathbf{f}_n\|,
\end{aligned}
\end{eqnarray*}
that is,
 \begin{eqnarray*}
\begin{aligned}
	\|(\mathbf{f}+\lambda^{-1}\mathbf{f}_n)-(\lambda \mathcal{A}\mathbf{f}+ \mathcal{A}\mathbf{f}_n)\|\geq\|\mathbf{f}+\lambda^{-1}\mathbf{f}_n\|.
\end{aligned}
\end{eqnarray*}
Letting $n\rightarrow\infty$ and then $\lambda\rightarrow 0$ gives $\|\mathbf{f}-\mathbf{h}\|\geq\|\mathbf{f}\|$, which is impossible obviously. Thus, we prove $\mathcal{A}$ is closable. 

For any $\mathbf{g}\in\mathrm{Dom}(\mathcal{A}^{\ast})$, we can compute
\[
\Re(\mathcal{A}^{\ast}\mathbf{g}, \mathbf{g})_{\mathcal{H}} = - \int_{\Gamma_{1}}\frac{d}{\rho}|g_4 |^2 \mathrm{d}\Gamma \leqslant 0,
\] 
which implies $\mathcal{A}^{\ast}$ is also dissipative.
From Corollary 4.4 in Chapter 1 of~\cite{Pazy}, it concludes that $\mathcal{A}$ is the infinitesimal generator of a $C_0$ semigroup of contractions on $\mathcal{H}$.
\end{proof}

In the sequel, we shall denote by $S^{\ast}(t)$ the adjoint of $S(t)$ and by $\mathcal{B}^{\ast}$ the adjoint of $\mathcal{B}$. By using~\cite[Corollary 10.6, pp. 41]{Pazy}, we know $S^{\ast}(t)$ is a strongly continuous semigroup of continuous linear operators and the infinitesimal generator of this semigroup is the operator $\mathcal{A}^\ast$.

Now, we recall the definition of the strong and weak solution of equation~\eqref{evolution} (see, for example,~\cite[Chapter 4.2, pp. 105-109]{Pazy}). 
A vector function $\mathbf{X}\in C^1([0, T]; \mathcal{H})$ is said to be a strong solution of equation~\eqref{evolution} on $[0, T]$ if $\mathbf{X}(t)\in\mathrm{Dom}(\mathcal{A})$, and~\eqref{evolution} is satisfied on $[0, T]$. We notice that $S(t-s)\mathcal{B}\mathbf{ u}$ is integrable with respect to $s$ if $\mathbf{ u} \in \mathcal{U}\cap L^2([0, T]; \mathcal{H})$. Then a vector function $\mathbf{X}\in C([0, T]; \mathcal{H})$ given by
\[
\mathbf{X}(t)= S(t)\mathbf{ x} + \int_{0}^{t}S(t-s)\mathcal{B}\mathbf{ u}\mathrm{d}s, \quad 0\leqslant t \leqslant T
\]
is well-defined, which is called a mild solution of equation~\eqref{evolution} on $[0, T]$. In addition, a vector function $\mathbf{X}\in C([0, T]; \mathcal{H})$ is a weak solution of equation~\eqref{evolution} on $[0, T]$ if for every ${\mathbf{ Z}}\in\mathrm{Dom}(\mathcal{A}^{\ast})$ the function $(\mathbf{X}(t), {\mathbf{Z}})_{\mathcal{H}}$ is absolutely continuous on $[0, T]$ and
\[
\frac{d}{dt}(\mathbf{X}(t), {\mathbf{Z}})_{\mathcal{H}} = (\mathbf{X}(t), \mathcal{A}^{\ast}{\mathbf{Z}})_{\mathcal{H}} +(\mathcal{B}\mathbf{ u}, {\mathbf{Z}})_{\mathcal{H}}, \quad \textrm{a.e. on $[0, T]$}.
\]
%\[
%\frac{d}{dt}\langle\mathbf{X}(t), \widetilde{\mathbf{ X}}\rangle = \langle\mathbf{X}(t), \mathcal{A}^{\ast}\widetilde{\mathbf{ X}}\rangle +\langle\mathcal{B}\mathbf{ u}, \widetilde{\mathbf{ X}}\rangle, \quad \textrm{a.e. on $[0, T]$}.
%\]
%Here, $\langle, \rangle$ is the pairing between $\mathcal{H}$ and its dual space $\mathcal{H}^\ast$. 
From Proposition~\ref{prop:operator} and the result in~\cite{ball1977strongly}, it follows that the mild solution is identical with a weak solution of equation~\eqref{evolution}. Obviously, strong solutions are also weak ones. 

In the context, a pair $(u, v)$ is considered a strong or weak solution to problem~\eqref{1}-\eqref{4} if they are the first two components of the vector function $\mathbf{X}(t)=\left(u, v, u_t, v_t\right)$, which is a solution to equation~\eqref{evolution} of the same type. Thus, one immediately gets that any strong solution $(u, v)$ of problem~\eqref{1}-\eqref{4} belongs to $C^1([0, T]; H^1(\Omega)\times H^1(\Gamma_1))\cap C^2([0, T]; L^2(\Omega)\times L^2(\Gamma_1))$, and the corresponding strong solution $\mathbf{X}(t)$ of equation~\eqref{evolution} is precisely $\left(u, v, u_t, v_t\right)$.

\subsection{Observability inequality}
\label{subsec:observability}

We are going to investigate the following homogeneous backward system of the evolution equation~\eqref{evolution}.
\begin{equation} \label{adjoint:abstract}
            \mathbf{Z}^{\prime}(t)= -\mathcal{A}^{\ast}(\mathbf{ Z}(t)),\quad t\in[0, T],\quad \mathbf{ Z}(T) = \mathbf{ z}\in\mathrm{Dom}(\mathcal{A}^{\ast})\subset\mathcal{H}.
\end{equation} 
So, setting $\mathbf{ Z}(t) = (\phi, \delta, \phi_t, \delta_t)$ and $\mathbf{ z}=(\phi_0, \delta_0, \phi_1, \delta_1)$, the backward system can be explicitly written as
\begin{equation} \label{adjoint}
\left
    \{
        \begin{array}{ll}
        		\phi_{tt} = c^2\triangle\phi, \quad \textrm{in $\mathbb{R}^{+} \times \Omega$},\\
            -\rho \phi_t= m \delta_{tt} - \mathrm{div}_{\Gamma}\left(\sigma\nabla_{\Gamma} \delta\right) - d \delta_t + k \delta, \quad \textrm{on $\mathbb{R}^{+} \times \Gamma_1$},\\ 
            \partial_{\mathbf{n}}\phi =\delta_t, \quad \textrm{on $\mathbb{R}^{+} \times \Gamma_1$}, \quad \partial_{\mathbf{n}}\phi = 0, \quad \textrm{on $\mathbb{R}^{+} \times \Gamma_0$},\\
            \left(\phi(T, \cdot), \delta(T, \cdot), \phi_t(T, \cdot), \delta_t(T, \cdot)\right) = (\phi_0, \delta_0, \phi_1, \delta_1), \quad \textrm{in $ \Omega$}.
        \end{array} 
\right.
\end{equation}

The following theorem presents a key estimate for the solution of system~\eqref{adjoint}, which is the core to prove our main results.
\begin{theorem}\label{observability}
Under the following geometrical assumptions on the boundary of the domain
\begin{equation} \label{G1}
	\Gamma_0 =\{x\in\Gamma: (x-x_0)\cdot\mathbf{n}\leq 0 \},\quad \Gamma_1 =\{x\in\Gamma: (x-x_0)\cdot\mathbf{n}\geq c \}
\end{equation}
for some point $x_0\in\mathbb{R}^3$ and some constant $c>0$. The strong solution $(\phi, \delta)$ of the backward system~\eqref{adjoint} satisfies
\begin{equation} \label{ineq:obs_abc}
	\|(\phi_0, \delta_0, \phi_1, \delta_1)\|^2_{\mathcal{H}} \leqslant C_{T}\int_{0}^{T}\left\{ \|\delta_t\|^2_{L^2(\Gamma_1)}+\|\delta_{tt}\|^2_{L^2(\Gamma_1)}\right\}\mathrm{d}t.
\end{equation}
\end{theorem}
\begin{proof}
The subsequent proof involves four steps.

\textbf{Step 1}. 
Let us introduce the following functional
\[
\mathcal{E}(s) := \frac{1}{2}\|( \phi(s), \delta(s), \phi_t(s), \delta_t(s))\|^2_{\mathcal{H}}.
\]
Then we have
\[
\mathcal{E}(T)=\mathcal{E}(s)+\frac{d}{\rho}\int_{s}^{T}\int_{\Gamma_1}\delta_{t}^2d\Gamma \mathrm{d}t
\]
and $\mathcal{E}(T)\geqslant \mathcal{E}(s)$ for every $s\in[0, T]$.
Moreover, from Theorem A.4.1 in~\cite{LTZ} the geometric condition~\eqref{G1} implies that there exists a vector field $h(x)$ and a scalar function $H(x)$
\[
	h(x):=\nabla H(x) \in[C^2(\overline{\Omega})]^3
\]
such that $\nabla H\cdot\mathbf{n}=0$ on $\Gamma_0$ and the Hessian matrix of $H$ evaluated on $\Gamma_0$ is positive definite, that is, $\nabla^2 H(x) \geqslant h_0 I$ for some constant $h_0 > 0$.
%Let $h$ be a $[C^{\infty}_{0}(\mathcal(R))]^{3}$-vector field, which will be given specifically later. 

Multiplying the first equation of~\eqref{adjoint} by $(h\cdot\nabla) \phi$ and integrating by parts, we have
\begin{equation}
\begin{split}\label{ob2}
	0&=\int_{\epsilon_0}^{T-\epsilon_0}\int_{\Omega}(h\cdot\nabla) \phi ( \phi_{tt}  -c^2\triangle \phi)\mathrm{d}x\mathrm{d}t \\
	&=\left[\int_{\Omega}(h\cdot\nabla) \phi  \phi_{t} dx\right]\bigg|_{\epsilon_0}^{T-\epsilon_0}-\int_{\epsilon_0}^{T-\epsilon_0}\int_{\Omega}\left[\nabla\cdot(\frac{h}{2} \phi_{t}^2)-\frac{\nabla\cdot h}{2} \phi_{t}^2)\right]\mathrm{d}x\mathrm{d}t\\
	&\qquad-c^2\int_{\epsilon_0}^{T-\epsilon_0}\int_{\Gamma}(h\cdot\nabla) \phi{\partial_{ \mathbf{n}}\phi}\mathrm{d}\Gamma\mathrm{ d}t+c^2\int_{\epsilon_0}^{T-\epsilon_0}\int_{\Omega}\nabla(h\cdot\nabla \phi)\cdot \nabla\phi \mathrm{d}x\mathrm{d}t
\end{split}
\end{equation}
for some $\epsilon_0\in(0, \frac{T}{2})$. Since we have
\begin{eqnarray*}
\begin{split}
	\nabla(h\cdot\nabla \phi)\cdot \nabla\phi &= \sum_{i, j}\frac{\partial h_j}{\partial x_i}\frac{\partial\phi}{\partial x_i}\frac{\partial\phi}{\partial x_j} + \sum_{i, j}h_j\frac{\partial}{\partial x_j}\left(\frac{\partial\phi}{\partial x_i}\right)^2\\
	&= \sum_{i, j}\frac{\partial h_j}{\partial x_i}\frac{\partial\phi}{\partial x_i}\frac{\partial\phi}{\partial x_j} + \nabla\cdot\left(\frac{h}{2}|\nabla\phi|^2\right)-\frac{\nabla\cdot h}{2}|\nabla\phi|^2,
\end{split}
\end{eqnarray*}
then using the boundary condition $\partial_{ \mathbf{n}}\phi = \delta_t$ on $\Gamma_1$ and $\partial_{ \mathbf{n}}\phi = 0$ on $\Gamma_0$, it deduces from~\eqref{ob2} that
\begin{equation}\label{ob2_1}
\begin{split}
	0&=\int_{\epsilon_0}^{T-\epsilon_0}\int_{\Omega}\frac{\nabla\cdot h}{2}[\phi_t^2-c^2|\nabla\phi|^2] \mathrm{d}x\mathrm{d}t-\int_{\epsilon_0}^{T-\epsilon_0}\int_{\Gamma}\left\{\frac{h\cdot\mathbf{n}}{2}[\phi_t^2-c^2|\nabla\phi|^2]\right\}\mathrm{d}\Gamma \mathrm{d}t\\
	&\qquad -c^2\int_{\epsilon_0}^{T-\epsilon_0}\int_{\Gamma_1}(h\cdot\nabla) \phi\delta_t \mathrm{d}\Gamma_1 \mathrm{d}t+c^2\int_{\epsilon_0}^{T-\epsilon_0}\int_{\Omega}\sum_{i, j}\frac{\partial h_j}{\partial x_i}\frac{\partial\phi}{\partial x_i}\frac{\partial\phi}{\partial x_j} \mathrm{d}x\mathrm{d}t\\
	&\qquad+\left[\int_{\Omega}(h\cdot\nabla) \phi \phi_t \mathrm{d}x\right]\bigg|_{\epsilon_0}^{T-\epsilon_0}.
\end{split}
\end{equation}
By the positive definite property of the Hessian matrix $\nabla^2 H(x)$, we get 
\begin{equation}\label{ob2_2}
\begin{split}
	c^2 h_{0}\int_{\epsilon_0}^{T-\epsilon_0}\int_{\Omega}|\nabla\phi|^2  \mathrm{d}x\mathrm{d}t \leqslant c^2\int_{\epsilon_0}^{T-\epsilon_0}\int_{\Omega}\sum_{i, j}\frac{\partial h_j}{\partial x_i}\frac{\partial\phi}{\partial x_i}\frac{\partial\phi}{\partial x_j} \mathrm{d}x\mathrm{d}t.
\end{split}
\end{equation}
Combining~\eqref{ob2_1} and~\eqref{ob2_2} gives
\begin{equation*}
\begin{aligned}
	&c^2 h_{0}\int_{\epsilon_0}^{T-\epsilon_0}\int_{\Omega}|\nabla\phi|^2  \mathrm{d}x\mathrm{d}t\\
	&\leqslant \int_{\epsilon_0}^{T-\epsilon_0}\int_{\Gamma_1}\left\{\frac{h\cdot\mathbf{n}}{2}[\phi_t^2-c^2|\nabla\phi|^2]\right\}\mathrm{d}\Gamma \mathrm{d}t - \int_{\epsilon_0}^{T-\epsilon_0}\int_{\Omega}\frac{\nabla\cdot h}{2}[\phi_t^2-c^2|\nabla\phi|^2]\mathrm{d}x\mathrm{d}t\\
	&\qquad +c^2\int_{\epsilon_0}^{T-\epsilon_0}\int_{\Gamma_1}(h\cdot\nabla) \phi\delta_t \mathrm{d}\Gamma \mathrm{d}t-\left[\int_{\Omega}(h\cdot\nabla) \phi \phi_t \mathrm{d}x\right]\bigg|_{\epsilon_0}^{T-\epsilon_0}.
\end{aligned}
\end{equation*}
Due to $|\nabla\phi|^2=|\partial_{ \mathbf{n}}\phi|^2 + |\nabla_{\Gamma}\phi|^2=|\delta_t|^2 + |\nabla_{\Gamma}\phi|^2$ on $\Gamma_1$, we further obtain the following estimate
\begin{equation}
\begin{aligned} \label{ob3}
	c^2 h_{0}\int_{\epsilon_0}^{T-\epsilon_0}\int_{\Omega}|\nabla\phi|^2  \mathrm{d}x\mathrm{d}t\leqslant& C(h)\left\{\int_{\epsilon_0}^{T-\epsilon_0}\int_{\Gamma_1}[\phi_t^2 + |\nabla_{\Gamma}\phi|^2]\mathrm{d}\Gamma \mathrm{d}t+\int_{0}^{T}\int_{\Gamma_1}\delta_t^2\mathrm{d}\Gamma \mathrm{d}t \right\}\\
	&+ \left|\int_{\epsilon_0}^{T-\epsilon_0}\int_{\Omega}\frac{\nabla\cdot h}{2}[\phi_t^2-c^2|\nabla\phi|^2] \mathrm{d}x\mathrm{d}t\right| +C(h)\mathcal{E}(T)
\end{aligned}
\end{equation}
Here, $\nabla_{\Gamma}$ is the tangential derivative on the boundary.

Multiplying the first equation of~\eqref{adjoint} by $\phi(\nabla\cdot h)$ and integrating by parts, we have
\begin{eqnarray*}
\begin{aligned}
	0&=\int_{\epsilon_0}^{T-\epsilon_0}\int_{\Omega}\phi(\nabla\cdot h) (\phi_{tt}-c^2\triangle \phi)\mathrm{d}x\mathrm{d}t \\
	&=\left[\int_{D}\phi(\nabla\cdot h) \phi_{t}dx\right]\bigg|_{\epsilon_0}^{T-\epsilon_0}-\int_{\epsilon_0}^{T-\epsilon_0}\int_{\Omega}\left[(\nabla\cdot h)\phi_t^2-c^2 \nabla( \phi \nabla\cdot h)\cdot \nabla\phi\right]\mathrm{d}x\mathrm{d}t\\
	&\qquad-c^2\int_{\epsilon_0}^{T-\epsilon_0}\int_{\Gamma} \phi(\nabla\cdot h){\partial_{ \mathbf{n}}\phi} \mathrm{d}\Gamma \mathrm{d}t\\
	&= \left[\int_{\Omega}\phi(\nabla\cdot h) \phi_{t}\mathrm{d}x \right]\bigg|_{\epsilon_0}^{T-\epsilon_0}-\int_{\epsilon_0}^{T-\epsilon_0}\int_{\Omega}(\nabla\cdot h)\left[\phi_t^2-c^2 |\nabla\phi|^2\right] \mathrm{d}x\mathrm{d}t\\
	&\qquad+c^2 \int_{\epsilon_0}^{T-\epsilon_0}\int_{\Omega}  \phi \nabla\phi \cdot \nabla(\nabla\cdot h)\mathrm{d}x\mathrm{d}t - c^2\int_{\epsilon_0}^{T-\epsilon_0}\int_{\Gamma_1} \phi(\nabla\cdot h)\delta_t \mathrm{d}\Gamma \mathrm{d}t,
\end{aligned}
\end{eqnarray*}
which implies
\begin{equation}
\begin{aligned} \label{ob4}
	&\int_{\epsilon_0}^{T-\epsilon_0}\int_{\Omega}(\nabla\cdot h)[\phi_t^2-c^2 |\nabla\phi|^2]\mathrm{d}x\mathrm{d}t\\
	&=\langle\phi(\nabla\cdot h), \phi_{t}\rangle\bigg|_{\epsilon_0}^{T-\epsilon_0}+c^2 \int_{\epsilon_0}^{T-\epsilon_0}\int_{\Omega}  \phi \nabla\phi \cdot \nabla(\nabla\cdot h)\mathrm{d}x\mathrm{d}t - c^2\int_{\epsilon_0}^{T-\epsilon_0}\int_{\Gamma_1} \phi(\nabla\cdot h)\delta_t \mathrm{d}\Gamma \mathrm{d}t.	
\end{aligned}
\end{equation}
Here, $\langle\cdot, \cdot\rangle$ means the the product in $H^{-\eta}(\Omega)\times H^{\eta}(\Omega)$ for $\eta>0$. 

Now, we introduce the standard denotation for the terms which are below the level of energy, that is,
\[
	\textrm{LOT}(\Phi, \Psi):=\|(\Phi, \Psi)\|^2_{C([0, T]; \mathcal{M})}.
\]
Here, $\Phi=\{\phi, \delta\}$, $\Psi=\{\phi_t, \delta_t\}$ and $\mathcal{M}=H^{1-\eta}(\Omega)\times H^{-\eta}(\Gamma_1)\times H^{-\eta}(\Omega)\times H^{-\eta}(\Gamma_1)$. Then by~\eqref{ob4} and the Young's inequality, we get
\begin{equation} \label{ob5}
\begin{split} 
	&\int_{\epsilon_0}^{T-\epsilon_0}\int_{\Omega}(\nabla\cdot h)[\phi_t^2-c^2 |\nabla\phi|^2]\mathrm{d}x \mathrm{d}t\\
	& \leq \tau_1\int_{\epsilon_0}^{T-\epsilon_0}\int_{\Omega}|\nabla\phi|^2 \mathrm{d}x \mathrm{d}t+C(\tau_1) \int_{0}^{T}\int_{\Gamma_1}\delta_t^2\mathrm{d}\Gamma \mathrm{d}t +\textrm{LOT}(\Phi, \Psi)
\end{split}
\end{equation}
for some small constant $0<\tau_1\ll1$.
From~\eqref{ob3} and~\eqref{ob5} it follows that
\begin{equation} 
\begin{split}\label{ob6}
	&\int_{\epsilon_0}^{T-\epsilon_0}\int_{\Omega}[c^{-2}|\phi_t|^2+ |\nabla\phi|^2]\mathrm{d}x \mathrm{d}t \\
	&\leq C(\tau_1, h) \left\{\int_{0}^{T}\int_{\Gamma_1}\delta_t^2\mathrm{d}\Gamma \mathrm{d}t + \int_{\epsilon_0}^{T-\epsilon_0}\int_{\Gamma_1}[\phi_t^2 + |\nabla_{\Gamma}\phi|^2]\mathrm{d}\Gamma \mathrm{d}t \right\} +C(h)\mathcal{E}(T)+\textrm{LOT}(\Phi, \Psi).
\end{split}
\end{equation}

\textbf{Step 2}. Let $\zeta(\cdot)\in C_{0}^{\infty}(\mathbb{R})$ be cut-off function given by
\begin{equation*} 
\zeta(t)=\left
    \{
        \begin{array}{ll}
            1,\quad &t\in[\epsilon_0, T-\epsilon_0],\\
            \textrm{a $C^{\infty}$ function with range in $(0, 1)$}, \quad &t\in(0, \epsilon_0)\cup (T-\epsilon_0, T),\\
            0, \quad &t\in(-\infty, 0]\cup [T, \infty).    
        \end{array} 
\right.
\end{equation*}
Using Lemma 7.2 in~\cite{LT}, we have
\begin{equation} \label{ob7}
\begin{aligned}
	&\int_{\epsilon_0}^{T-\epsilon_0}\int_{\Gamma_1} |\nabla_{\Gamma}\phi|^2 \mathrm{d}\Gamma \mathrm{d}t\\
	&\leq \int_{0}^{T}\int_{\Gamma_1} |\nabla_{\Gamma}(\zeta\phi)|^2 \mathrm{d}\Gamma \mathrm{d}t\\
	&\leq C(T, \epsilon_0)\left\{\int_{0}^{T}\int_{\Gamma_1} |{\partial_ \mathbf{n}\phi}|^2 \mathrm{d}\Gamma \mathrm{d}t +\int_{0}^{T}\int_{\Gamma_1}| \frac{\partial}{\partial t}(\zeta \phi)|^2 \mathrm{d}\Gamma \mathrm{d}t\right\}+\textrm{LOT}(\Phi, \Psi)\\
	&\leq C(T, \epsilon_0)\left\{\int_{0}^{T}\int_{\Gamma_1} |\delta_t|^2 \mathrm{d}\Gamma \mathrm{d}t +\int_{0}^{T}\int_{\Gamma_1} \zeta^2|\phi_t|^2 \mathrm{d}\Gamma \mathrm{d}t\right\}+\textrm{LOT}(\Phi, \Psi).
\end{aligned}
\end{equation}

By~\eqref{ob6} and~\eqref{ob7} we have
\begin{equation}\label{ob8}
\begin{split}
\int_{\epsilon_0}^{T-\epsilon_0}\mathcal{E}(s)\mathrm{d}s \leqslant &C(T, \epsilon_0, \tau_1, h, m, \rho)\left\{\int_{0}^{T}\int_{\Gamma_1} |\delta_t|^2 \mathrm{d}\Gamma \mathrm{d}t +\int_{0}^{T}\int_{\Gamma_1}\zeta^2|\phi_t|^2 \mathrm{d}\Gamma \mathrm{d}t\right\}\\
&+ \int_{0}^{T}\int_{\Gamma_1}\left[\frac{k}{\rho} |\delta|^2 + \frac{\sigma}{\rho}|\nabla_{\Gamma}\delta|^2\right]\mathrm{d}\Gamma \mathrm{d}t+C(h)\mathcal{E}(T) +\textrm{LOT}(\Phi, \Psi).
\end{split}
\end{equation}

Multiplying the second equation of~\eqref{adjoint} by $\delta$ and integrating by parts, we obtain
\begin{equation}\label{ob9}
\begin{split}
0= &\int_{0}^{T}\int_{\Gamma_1}\delta\left[\rho\phi_t + m \delta_{tt} - \mathrm{div}_{\Gamma}\left(\sigma\nabla_{\Gamma} \delta\right) - d \delta_t + k \delta\right]\mathrm{d}\Gamma \mathrm{d}t\\
=&\left[ \int_{\Gamma_1}\rho\delta\phi\mathrm{d}\Gamma\right]_{t=0}^{t=T} -\int_{0}^{T}\int_{\Gamma_1}\rho\delta_t\phi \mathrm{d}\Gamma \mathrm{d}t +\left[ \int_{\Gamma_1}m\delta\delta_t\mathrm{d}\Gamma\right]_{t=0}^{t=T} -\int_{0}^{T}\int_{\Gamma_1}m\delta^2_t \mathrm{d}\Gamma \mathrm{d}t\\
&-\left[ \int_{\Gamma_1}\frac{d}{2}\delta^2\mathrm{d}\Gamma\right]_{t=0}^{t=T} +\int_{0}^{T}\int_{\Gamma_1}\left[k \delta^2 + \sigma|\nabla_{\Gamma} \delta|^2\right]\mathrm{d}\Gamma \mathrm{d}t.
\end{split}
\end{equation}
Using the Young's inequality and the Poincar\'e inequality, we get
\[
\left|\int_{0}^{T}\int_{\Gamma_1}\rho\delta_t\phi \mathrm{d}\Gamma \mathrm{d}t\right|\leqslant\rho\tau_2 T\mathcal{E}(T) + C(\rho, \tau_2)\int_{0}^{T}\int_{\Gamma_1}|\delta_t|^2 \mathrm{d}\Gamma \mathrm{d}t
\]
for some small constant $0<\tau_2\ll1$. We also note that
\[
\left| \left[ \int_{\Gamma_1}\rho\delta\phi\mathrm{d}\Gamma\right]_{t=0}^{t=T} + \left[ \int_{\Gamma_1}m\delta\delta_t\mathrm{d}\Gamma\right]_{t=0}^{t=T} -\left[ \int_{\Gamma_1}\frac{d}{2}\delta^2\mathrm{d}\Gamma\right]_{t=0}^{t=T}\right|\leqslant C(\rho, m, d)\mathcal{E}(T).
\]
Then by~\eqref{ob9} we have
\begin{equation}\label{ob9_1}
\int_{0}^{T}\int_{\Gamma_1}\left[k \delta^2 + \sigma|\nabla_{\Gamma} \delta|^2\right]\mathrm{d}\Gamma \mathrm{d}t\leqslant \left(C(\rho, m, d)+\rho\tau_2 T\right)\mathcal{E}(T) + C(\rho, \tau_2)\int_{0}^{T}\int_{\Gamma_1}|\delta_t|^2 \mathrm{d}\Gamma \mathrm{d}t.
\end{equation}
Combining this estimate with~\eqref{ob8} derives
\begin{equation}\label{ob10}
\begin{split}
\int_{\epsilon_0}^{T-\epsilon_0}\mathcal{E}(s)\mathrm{d}s \leqslant &C(T, \epsilon_0, \tau_1, \tau_2, h, m, \rho)\left\{\int_{0}^{T}\int_{\Gamma_1} |\delta_t|^2 \mathrm{d}\Gamma \mathrm{d}t +\int_{0}^{T}\int_{\Gamma_1}\zeta^2|\phi_t|^2 \mathrm{d}\Gamma \mathrm{d}t\right\}\\
&+\left(C(h, \rho, m, d)+\rho\tau_2 T\right)\mathcal{E}(T) +\textrm{LOT}(\Phi, \Psi).
\end{split}
\end{equation}

\textbf{Step 3}. In this step, we are going to estimate $\int_{0}^{T}\int_{\Gamma_1}\zeta^2|\phi_t|^2 \mathrm{d}\Gamma \mathrm{d}t$. 
%To do so, we set
%\[
%	(\tilde{\phi}, \tilde{\delta}):= (\zeta\phi, \zeta\delta).
%\]
%Thus, we know that the couple $(\tilde{\phi}, \tilde{\delta})$ satisfies
%\begin{equation*} 
%\left
%    \{
%        \begin{array}{ll}
%        		\tilde{\phi}_{tt} = c^2\triangle\tilde{\phi} +\zeta_{tt}\phi + 2\zeta_t \phi_t, \quad \textrm{in $\mathbb{R}^{+} \times \Omega$},\\
%            -\rho \tilde{\phi}_t= m \tilde{\delta}_{tt} - \mathrm{div}_{\Gamma}\left(\sigma\nabla_{\Gamma} \tilde{\delta}\right) - d \tilde{\delta}_t+ k \tilde{\delta}\\
%            \qquad\qquad -\rho\zeta_t\phi -2m\zeta_t\delta_t -m\zeta_{tt}\delta + d\zeta_t\delta, \quad \textrm{on $\mathbb{R}^{+} \times \Gamma_1$},\\ 
%            \partial_{\mathbf{n}}\tilde{\phi }=\zeta\delta_t, \quad \textrm{on $\mathbb{R}^{+} \times \Gamma_1$}, \quad \partial_{\mathbf{n}}\tilde{\phi} = 0, \quad \textrm{on $\mathbb{R}^{+} \times \Gamma_0$}.
%        \end{array} 
%\right.
%\end{equation*}
Note that the second equation of~\eqref{adjoint} yields
 \begin{equation}\label{ob11_1}
\begin{split}
&\rho^2\int_{0}^{T}\int_{\Gamma_1}\zeta^2|\phi_t|^2 \mathrm{d}\Gamma \mathrm{d}t \\
&= \int_{0}^{T}\int_{\Gamma_1}\left[m\zeta \delta_{tt} - \zeta\mathrm{div}_{\Gamma}\left(\sigma\nabla_{\Gamma} \delta\right) - d \zeta\delta_t + k\zeta \delta\right]^2\mathrm{d}\Gamma \mathrm{d}t \\
&\leqslant C(m, d)\int_{0}^{T}\int_{\Gamma_1}\left[|\zeta \delta_{tt}|^2 +|\zeta\delta_t|^2 \right]\mathrm{d}\Gamma \mathrm{d}t +\int_{0}^{T}\int_{\Gamma_1}\left[ {|\zeta\mathrm{div}_{\Gamma}\left(\sigma\nabla_{\Gamma} \delta\right)|^2} + {k^2|\zeta \delta|^2} \right]\mathrm{d}\Gamma \mathrm{d}t.
\end{split}
\end{equation}
Moreover, we have
\begin{equation*}
\begin{split}
0= &\int_{0}^{T}\int_{\Gamma_1}\zeta^2\delta\left[\rho\phi_t + m \delta_{tt} - \mathrm{div}_{\Gamma}\left(\sigma\nabla_{\Gamma} \delta\right) - d \delta_t + k \delta\right]\mathrm{d}\Gamma \mathrm{d}t\\
=& -\int_{0}^{T}\int_{\Gamma_1}\rho(\zeta^2\delta)_t\phi \mathrm{d}\Gamma \mathrm{d}t  + \int_{0}^{T}\int_{\Gamma_1}m\zeta^2\delta\delta_{tt} \mathrm{d}\Gamma \mathrm{d}t - \int_{0}^{T}\int_{\Gamma_1}d\zeta^2\delta_t\delta \mathrm{d}\Gamma \mathrm{d}t\\
&+\int_{0}^{T}\int_{\Gamma_1}\left[k |\zeta\delta|^2 + \sigma|\nabla_{\Gamma} (\zeta\delta)|^2\right]\mathrm{d}\Gamma \mathrm{d}t,
\end{split}
\end{equation*}
which implies
\begin{equation*}
\begin{split}
&\int_{0}^{T}\int_{\Gamma_1}k^2 |\zeta\delta|^2 \mathrm{d}\Gamma \mathrm{d}t\\
&\leqslant \int_{0}^{T}\int_{\Gamma_1}k\rho(\zeta^2\delta)_t\phi \mathrm{d}\Gamma \mathrm{d}t  - \int_{0}^{T}\int_{\Gamma_1}km\zeta^2\delta\delta_{tt} \mathrm{d}\Gamma \mathrm{d}t + \int_{0}^{T}\int_{\Gamma_1}d\zeta^2\delta_t\delta \mathrm{d}\Gamma \mathrm{d}t\\
&\leqslant \int_{0}^{T}\int_{\Gamma_1}k\rho(2\zeta\zeta_t\delta\phi) \mathrm{d}\Gamma \mathrm{d}t +\frac{1}{2}\int_{0}^{T}\int_{\Gamma_1}k\rho|\zeta\phi|^2  \mathrm{d}\Gamma \mathrm{d}t  +\frac{5}{8} \int_{0}^{T}\int_{\Gamma_1}k^2|\zeta\delta|^2 \mathrm{d}\Gamma \mathrm{d}t \\
&\quad +\frac{1}{2}\int_{0}^{T}\int_{\Gamma_1}\left[\left(k\rho + 2\frac{d^2}{k^2}\right) |\zeta\delta_t|^2 + m^2|\zeta\delta_{tt}|^2\right]\mathrm{d}\Gamma \mathrm{d}t\\
&\leqslant \int_{0}^{T}\int_{\Gamma_1}\left(8\rho^2|\zeta_t\phi|^2 +\frac{1}{2} k\rho|\zeta\phi|^2  \right)\mathrm{d}\Gamma \mathrm{d}t  +\frac{3}{4} \int_{0}^{T}\int_{\Gamma_1}k^2|\zeta\delta|^2 \mathrm{d}\Gamma \mathrm{d}t
\\
&\quad +\frac{1}{2}\int_{0}^{T}\int_{\Gamma_1}\left[\left(k\rho + 2\frac{d^2}{k^2}\right) |\zeta\delta_t|^2 + m^2|\zeta\delta_{tt}|^2\right]\mathrm{d}\Gamma \mathrm{d}t.
\end{split}
\end{equation*}
This can further yields
\begin{equation} \label{ob11_2}
\begin{split}
&\int_{0}^{T}\int_{\Gamma_1}k^2 |\zeta\delta|^2 \mathrm{d}\Gamma \mathrm{d}t\\
&\leqslant \int_{0}^{T}\int_{\Gamma_1}\left(32\rho^2|\zeta_t\phi|^2 + 2k\rho|\zeta\phi|^2\right) \mathrm{d}\Gamma \mathrm{d}t + 2\int_{0}^{T}\int_{\Gamma_1}\left[\left(k\rho + 2\frac{d^2}{k^2}\right) |\zeta\delta_t|^2 + m^2|\zeta\delta_{tt}|^2\right]\mathrm{d}\Gamma \mathrm{d}t .
\end{split}
\end{equation}
From~\eqref{ob11_1} and~\eqref{ob11_2}, we can deduce
 \begin{equation}\label{ob12}
\begin{split}
&\rho^2\int_{0}^{T}\int_{\Gamma_1}\zeta^2|\phi_t|^2 \mathrm{d}\Gamma \mathrm{d}t \\
&\leqslant C(\rho, m, d, k)\int_{0}^{T}\int_{\Gamma_1}\left[|\zeta \delta_{tt}|^2 +|\delta_t|^2 \right]\mathrm{d}\Gamma \mathrm{d}t +\int_{0}^{T}\int_{\Gamma_1}|\zeta\mathrm{div}_{\Gamma}\left(\sigma\nabla_{\Gamma} \delta\right)|^2 \mathrm{d}\Gamma \mathrm{d}t\\
&\quad + \int_{0}^{T}\int_{\Gamma_1}\left(32\rho^2|\zeta_t\phi|^2 + 2k\rho|\zeta\phi|^2\right) \mathrm{d}\Gamma \mathrm{d}t.
\end{split}
\end{equation}

Using the Young's inequality and integrating by parts, one can derive from the second equation of~\eqref{adjoint} that
 \begin{equation*}
\begin{split}
&\int_{0}^{T}\int_{\Gamma_1}{|\zeta\mathrm{div}_{\Gamma}\left(\sigma\nabla_{\Gamma} {\delta}\right)|^2} \mathrm{d}\Gamma \mathrm{d}t\\
&=\int_{0}^{T}\int_{\Gamma_1}\left(\rho \zeta{\phi}_t+m \zeta{\delta}_{tt}- d \zeta{\delta}_t+ k \zeta{\delta} \right)\zeta\mathrm{div}_{\Gamma}\left(\sigma\nabla_{\Gamma} {\delta}\right) \mathrm{d}\Gamma \mathrm{d}t\\
&=\int_{0}^{T}\int_{\Gamma_1} \rho \zeta^2 {\phi}_t \mathrm{div}_{\Gamma}\left(\sigma\nabla_{\Gamma}{\delta}\right) \mathrm{d}\Gamma \mathrm{d}t -\int_{0}^{T}\int_{\Gamma_1}k\sigma|\zeta\nabla_{\Gamma} {\delta}|^2\mathrm{d}\Gamma \mathrm{d}t\\
&\quad+ \int_{0}^{T}\int_{\Gamma_1}\left(m \zeta{\delta}_{tt}- d \zeta{\delta}_t \right)\zeta \mathrm{div}_{\Gamma}\left(\sigma\nabla_{\Gamma} {\delta}\right) \mathrm{d}\Gamma \mathrm{d}t \\
&\leqslant \int_{0}^{T}\int_{\Gamma_1} \rho \zeta^2 {\phi}_t \mathrm{div}_{\Gamma}\left(\sigma\nabla_{\Gamma} {\delta}\right) \mathrm{d}\Gamma \mathrm{d}t +\tau_3\int_{0}^{T}\int_{\Gamma_1}|\zeta\mathrm{div}_{\Gamma}\left(\sigma\nabla_{\Gamma} {\delta}\right)|^2 \mathrm{d}\Gamma \mathrm{d}t\\
&\quad+ C(m, d, \tau_3)\int_{0}^{T}\int_{\Gamma_1}\left[|\zeta \delta_{tt}|^2 +|\zeta\delta_t|^2 \right]\mathrm{d}\Gamma \mathrm{d}t 
\end{split}
\end{equation*}
from which it follows that
 \begin{equation}\label{ob13}
\begin{split}
&(1-\tau_3)\int_{0}^{T}\int_{\Gamma_1}{|\zeta\mathrm{div}_{\Gamma}\left(\sigma\nabla_{\Gamma} {\delta}\right)|^2} \mathrm{d}\Gamma \mathrm{d}t\\
&\leqslant \int_{0}^{T}\int_{\Gamma_1} {\rho \zeta^2 {\phi}_t \mathrm{div}_{\Gamma}\left(\sigma\nabla_{\Gamma} {\delta}\right)} \mathrm{d}\Gamma \mathrm{d}t+ C(m, d, \tau_3)\int_{0}^{T}\int_{\Gamma_1}\left[|\zeta \delta_{tt}|^2 +|\zeta\delta_t|^2 \right]\mathrm{d}\Gamma \mathrm{d}t. 
\end{split}
\end{equation}
By integrating by parts and the Young's inequality, we get
\begin{equation*}
\begin{split}
&\int_{0}^{T}\int_{\Gamma_1} {\rho \zeta^2 {\phi}_t \mathrm{div}_{\Gamma}\left(\sigma\nabla_{\Gamma} {\delta}\right)} \mathrm{d}\Gamma \mathrm{d}t \\
&=-\int_{0}^{T}\int_{\Gamma_1} \rho {\phi}\left[ 2\zeta\zeta_t\mathrm{div}_{\Gamma}\left(\sigma\nabla_{\Gamma} {\delta}\right) + \zeta^2\mathrm{div}_{\Gamma}\left(\sigma\nabla_{\Gamma} {\delta}_t\right)\right] \mathrm{d}\Gamma \mathrm{d}t\\
&=-2\int_{0}^{T}\int_{\Gamma_1} \rho \zeta\zeta_t {\phi}\mathrm{div}_{\Gamma}\left(\sigma\nabla_{\Gamma} {\delta}\right)\mathrm{d}\Gamma \mathrm{d}t + \int_{0}^{T}\int_{\Gamma_1}\rho\sigma\zeta^2\nabla_{\Gamma}\phi\cdot \nabla_{\Gamma} {\delta}_t\mathrm{d}\Gamma \mathrm{d}t\\
&\leqslant \tau_4\int_{0}^{T}\int_{\Gamma_1} |\zeta\mathrm{div}_{\Gamma}\left(\sigma\nabla_{\Gamma} {\delta}\right)|^2\mathrm{d}\Gamma \mathrm{d}t  +C(\rho, \tau_4)\int_{0}^{T}\int_{\Gamma_1} {\phi}^2\mathrm{d}\Gamma \mathrm{d}t\\
&\quad +\tau_5{\int_{0}^{T}\int_{\Gamma_1}\sigma |\nabla_{\Gamma}(\zeta\phi)|^2\mathrm{d}\Gamma \mathrm{d}t} + C(\rho, \tau_5)\int_{0}^{T}\int_{\Gamma_1}\zeta^2\sigma|\nabla_{\Gamma} {\delta}_t|^2\mathrm{d}\Gamma \mathrm{d}t
\end{split}
\end{equation*}
and
\begin{equation*}
\begin{split}
\int_{0}^{T}\int_{\Gamma_1}\zeta^2\sigma|\nabla_{\Gamma} {\delta}_t|^2\mathrm{d}\Gamma \mathrm{d}t =&- \int_{0}^{T}\int_{\Gamma_1}\zeta^2{\delta}_t \mathrm{div}_{\Gamma}\left(\sigma\nabla_{\Gamma} {\delta}_t \right)\mathrm{d}\Gamma \mathrm{d}t\\
\leqslant &\tau_6\int_{0}^{T}\int_{\Gamma_1} |\zeta\mathrm{div}_{\Gamma}\left(\sigma\nabla_{\Gamma} {\delta}\right)|^2\mathrm{d}\Gamma \mathrm{d}t + C(\tau_6)\int_{0}^{T}\int_{\Gamma_1}|\zeta\delta_t|^2\mathrm{d}\Gamma \mathrm{d}t. 
\end{split}
\end{equation*}
Based on these two estimates and~\eqref{ob7}, one can derive
%\begin{equation}
%\begin{split}
%&\int_{0}^{T}\int_{\Gamma_1} \rho \zeta^2 {\phi}_t \mathrm{div}_{\Gamma}\left(\sigma\nabla_{\Gamma} {\delta}\right) \mathrm{d}\Gamma \mathrm{d}t \\
%&\leqslant (\tau_4 + \tau_6 C(\rho, \tau_5))\int_{0}^{T}\int_{\Gamma_1} |\zeta\mathrm{div}_{\Gamma}\left(\sigma\nabla_{\Gamma} {\delta}\right)|^2\mathrm{d}\Gamma \mathrm{d}t  +C(\rho, \tau_4)\int_{0}^{T}\int_{\Gamma_1} {\phi}^2\mathrm{d}\Gamma \mathrm{d}t\\
%&\quad +\tau_5\underbrace{\int_{0}^{T}\int_{\Gamma_1}\sigma |\nabla_{\Gamma}(\zeta\phi)|^2\mathrm{d}\Gamma \mathrm{d}t}_{\eqref{ob7}} + C(\tau_6)\int_{0}^{T}\int_{\Gamma_1}|\zeta\delta_t|^2\mathrm{d}\Gamma \mathrm{d}t.
%\end{split}
%\end{equation}
\begin{equation}\label{ob14}
\begin{split}
&\int_{0}^{T}\int_{\Gamma_1} {\rho \zeta^2 {\phi}_t \mathrm{div}_{\Gamma}\left(\sigma\nabla_{\Gamma} {\delta}\right)} \mathrm{d}\Gamma \mathrm{d}t \\
&\leqslant (\tau_4 + \tau_6 C(\rho, \tau_5))\int_{0}^{T}\int_{\Gamma_1} |\zeta\mathrm{div}_{\Gamma}\left(\sigma\nabla_{\Gamma} {\delta}\right)|^2\mathrm{d}\Gamma \mathrm{d}t  \\
&\quad +C(T, \epsilon_0, \sigma)\tau_5\int_{0}^{T}\int_{\Gamma_1} |\delta_t|^2 \mathrm{d}\Gamma \mathrm{d}t+ C(\rho, \tau_5, \tau_6)\int_{0}^{T}\int_{\Gamma_1}|\zeta\delta_t|^2\mathrm{d}\Gamma \mathrm{d}t \\
&\quad +C(T, \epsilon_0, \sigma)\tau_5\int_{0}^{T}\int_{\Gamma_1} \zeta^2|\phi_t|^2 \mathrm{d}\Gamma \mathrm{d}t+C(\rho, \tau_4)\int_{0}^{T}\int_{\Gamma_1} {\phi}^2\mathrm{d}\Gamma \mathrm{d}t + \textrm{LOT}(\Phi, \Psi).
\end{split}
\end{equation}
Substituting~\eqref{ob14} into~\eqref{ob13}, one has the following estimate
 \begin{equation}\label{ob15}
\begin{split}
&(1-\tau_3-\tau_4 - \tau_6 C(\rho, \tau_5))\int_{0}^{T}\int_{\Gamma_1}|\zeta\mathrm{div}_{\Gamma}\left(\sigma\nabla_{\Gamma} {\delta}\right)|^2 \mathrm{d}\Gamma \mathrm{d}t\\
&\leqslant C(T, \epsilon_0, \sigma)\tau_5\int_{0}^{T}\int_{\Gamma_1} \zeta^2|\phi_t|^2 \mathrm{d}\Gamma \mathrm{d}t+ C(\rho, m, d, \tau_3,  \tau_5, \tau_6)\int_{0}^{T}\int_{\Gamma_1}\left[|\zeta \delta_{tt}|^2 +|\zeta\delta_t|^2 \right]\mathrm{d}\Gamma \mathrm{d}t\\
&\quad+C(T, \epsilon_0, \sigma)\tau_5\int_{0}^{T}\int_{\Gamma_1} |\delta_t|^2 \mathrm{d}\Gamma \mathrm{d}t +C(\rho, \tau_4)\int_{0}^{T}\int_{\Gamma_1} {\phi}^2\mathrm{d}\Gamma \mathrm{d}t+\textrm{LOT}(\Phi, \Psi).
\end{split}
\end{equation}
Let $\frac{1}{N_1}: = 1-\tau_3-\tau_4 - \tau_6 C(\rho, \tau_5)$. Choosing $\tau_i$, $i=3, 4, 5, 6$ small enough ensures 
\[
0<\frac{1}{N_2}:=\frac{\rho^2}{N_1} - C(T, \epsilon_0, \sigma)\tau_5<1.
\]
Combining~\eqref{ob12} with~\eqref{ob15} gives
 \begin{equation}\label{ob16}
\begin{split}
&\int_{0}^{T}\int_{\Gamma_1}\zeta^2|\phi_t|^2 \mathrm{d}\Gamma \mathrm{d}t \\
%&\leqslant N_2\left[\frac{C(\rho, m, d, k)}{N_1} + + C(\rho, m, d, \sigma, \tau_3,  \tau_5, \tau_6)\right]\int_{0}^{T}\int_{\Gamma_1}\left[|\zeta \delta_{tt}|^2 +|\delta_t|^2 \right]\mathrm{d}\Gamma \mathrm{d}t \\
%&\quad + \int_{0}^{T}\int_{\Gamma_1}\frac{N_2}{N_1}\left(16\rho^2|\zeta_t\phi|^2 + 2k\rho|\zeta\phi|^2\right) \mathrm{d}\Gamma \mathrm{d}t + N_2C(\rho, \tau_4)\int_{0}^{T}\int_{\Gamma_1} {\phi}^2\mathrm{d}\Gamma \mathrm{d}t+N_2 \textrm{LOT}(\Phi, \Psi). 
&\leqslant N_2\left[\frac{C(\rho, m, d, k)}{N_1} + C(\rho, m, d, \sigma, \tau_3,  \tau_5, \tau_6)\right]\int_{0}^{T}\int_{\Gamma_1}\left[|\zeta \delta_{tt}|^2 +|\delta_t|^2 \right]\mathrm{d}\Gamma \mathrm{d}t \\
&\quad+\left[ T\left(\frac{N_2}{N_1}\left(32\rho^2+2k\rho\right)\right)+N_2C(\rho, \tau_4)+N_2\right] \textrm{LOT}(\Phi, \Psi).
\end{split}
\end{equation}

\textbf{Step 4}. Now, one can deduce from~\eqref{ob10} and~\eqref{ob16} that
\begin{equation}\label{step4_00}
\begin{split}
\int_{\epsilon_0}^{T-\epsilon_0}\mathcal{E}(s)\mathrm{d}s \leqslant &C(T, \epsilon_0, \tau_i, h, m, \rho, d, k, N_1, N_2)\int_{0}^{T}\int_{\Gamma_1}\left[|\zeta \delta_{tt}|^2 +|\delta_t|^2 \right]\mathrm{d}\Gamma \mathrm{d}t \\
&+\left(C(h, \rho, m, d)+\tau_2 T\right)\mathcal{E}(T) +C(T, \rho, k, \tau_4, N_1, N_2)\textrm{LOT}(\Phi, \Psi).
\end{split}
\end{equation}
Here, $i=3, 4, 5, 6$.
Due to
\begin{equation*}
(T-2\epsilon_0)\mathcal{E}(T) = \int_{\epsilon_0}^{T-\epsilon_0}\left(\mathcal{E}(s) +\frac{d}{\rho}\int_{s}^{T}\int_{\Gamma_1}\delta_{t}^2d\Gamma \mathrm{d}t\right)\mathrm{d}s,
\end{equation*}
then by~\eqref{step4_00} we get
\begin{equation}\label{step4_01}
\begin{split}
&(T-2\epsilon_0 - C(h, \rho, m, d)+\tau_2 T)\mathcal{E}(T) \\
&\leqslant C(T, \epsilon_0, \tau_i, h, m, \rho, d, k, N_1, N_2)\int_{0}^{T}\int_{\Gamma_1}\left[|\zeta \delta_{tt}|^2 +|\delta_t|^2 \right]\mathrm{d}\Gamma \mathrm{d}t \\
&\quad +C(T, \rho, k, \tau_4, N_1, N_2)\textrm{LOT}(\Phi, \Psi).
\end{split}
\end{equation}

Subsequently, we shall show that there exists a constant $C_T>0$ such that the following inequality holds
\begin{equation}\label{step4_1}
	\textrm{LOT}(\Phi, \Psi)\leq C_T \|\delta_{t}\|^2_{L^2([0, T]\times \Gamma_1)}.
\end{equation}
To do so, we use the method of contradiction. 
Suppose that the claim~\eqref{step4_1} is false. Then there exists a sequence of strong solutions $\{\Phi^{(n)}(t), \Psi^{(n)}(t)\}$ to the backward system~\eqref{adjoint} with the terminal data 
$\{\Phi^{(n)}(T), \Psi^{(n)}(T)\}\subset\mathcal{H}$ such that
\begin{equation} \label{false}
\begin{aligned}
	\textrm{LOT}(\Phi^{(n)}, \Psi^{(n)}) =1, \quad \|\delta^{(n)}_{t}\|^2_{L^2([0, T]\times \Gamma_1)}\xrightarrow{n\rightarrow \infty} 0.
\end{aligned}
\end{equation}

From~\eqref{step4_01}, it follows that $\|\{\Phi^{(n)}(T), \Psi^{(n)}(T)\}\|_{\mathcal{H}}$ is bounded. Then there exists a subsequence, still denoted by $\{\Phi^{(n)}(T), \Psi^{(n)}(T)\}$, and $(\hat{\Phi}(T), \hat{\Psi}(T))\in\mathcal{H}$ such that
\begin{eqnarray*}
	\{\Phi^{(n)}(T), \Psi^{(n)}(T)\}\xrightarrow[n\rightarrow\infty]{weakly} \{\hat{\Phi}(T), \hat{\Psi}(T)\},\quad \textrm{in $\mathcal{H}$}.
\end{eqnarray*}
Let $\{\hat{\Phi}(t), \hat{\Psi}(t)\}$ be the strong solutions of~\eqref{adjoint} subject to the terninal data $\{\hat{\Phi}(T), \hat{\Psi}(T)\}$. Due to fact that $\mathcal{E}(T)\geqslant \mathcal{E}(s)$, one can see that $\|\{\Phi^{(n)}(t), \Psi^{(n)}(t)\}\|_{C([0, T]; \mathcal{H})}$ is bounded. Then we have
\begin{equation}\label{step4_2}
	\{\Phi^{(n)}(t), \Psi^{(n)}(t)\}\xrightarrow{\textrm{weak star}} \{\hat{\Phi}(t), \hat{\Psi}(t)\},\quad \textrm{in $L^{\infty}([0, T]; \mathcal{H})$}.
\end{equation}
Let $\mathcal{N}=H^{-\eta}(\Omega)\times H^{-\eta}(\Gamma_1)\times (H^{1}(\Omega))^{\prime}\times H^{-\eta}(\Gamma_1)$. For any $\omega\in H^1$ and $t\in(0, T)$, we have the following equation holds true
\[
 (\phi^{(n)}_{tt}, \omega)_{H^{-1}\times H^1} =\int_{\Omega}\Delta \phi^{(n)} \omega dx = \int_{\Gamma_1}\delta^{(n)}_t\omega d\Gamma_1-\int_{\Omega}\nabla \phi^{(n)}\cdot \nabla\omega dx, 
 \]
 which implies that  
 \[
 	|(\phi^{(n)}_{tt}, \omega)_{H^{-1}\times H^1}|\leq C(\|\delta^{(n)}\|_{L^2}, \|\nabla\phi^{(n)}\|_{L^2})\|\omega\|_{H^1}.
\]
Hence, $\phi^{(n)}_{tt}\in L^{\infty}([0, T], (H^{1}(\Omega))^{\prime})$. Thus, it follows that $\|\{\Phi^{(n)}(t), \Psi^{(n)}(t)\}\|_{L^{\infty}([0, T], \mathcal{N})}$ is bounded uniformly. Due to $\mathcal{H}\subset\subset\mathcal{M}\subset\mathcal{N}$, we deduce from Aubin's theorem \cite{Simon} that
 \begin{eqnarray*}
	\{\Phi^{(n)}(t), \Psi^{(n)}(t)\}\xrightarrow{\textrm{strongly}} \{\hat{\Phi}(t), \hat{\Psi}(t)\},\quad \textrm{in $L^{\infty}([0, T]; \mathcal{M})$},
\end{eqnarray*} 
which gives
\begin{equation}\label{step4_3}
\|\{\hat{\Phi}(t), \hat{\Psi}(t)\}\|_{C([0, T]; \mathcal{M})}=\textrm{LOT}(\Phi^{(n)}, \Psi^{(n)}) =1.
\end{equation} 
By~\eqref{false} and~\eqref{step4_2}, we have $\hat{\delta}_t =0$. Let $\Theta(t, x)=\hat{\phi}_{tt}$. Then we have
\begin{equation*}
\left
    \{
        \begin{array}{ll}
             \Theta_{tt} = c^2\triangle \Theta, \quad \textrm{in $[0, T] \times \Omega$,}\\
            \frac{\partial \Theta}{\partial \mathbf{n}}= 0, \quad \textrm{on $[0, T] \times \Gamma_0$,}\\
            \Theta = 0,  \quad \textrm{on $[0, T] \times \Gamma_1$}.   
        \end{array} 
\right.
\end{equation*}
By Holmgren's Uniqueness Theorem, choosing $T>2\textrm{diam}(\Omega)$, it follows that $\Theta=\hat{\phi}_{tt}=0$. And then we obtain
\begin{equation*}
\left
    \{
        \begin{array}{ll}
             \triangle\hat{\phi} = 0, \quad \textrm{in $ \Omega$,}\\
            \frac{\partial \hat{\phi}}{\partial \mathbf{n}}= 0, \quad \textrm{on $ \Gamma$,} 
        \end{array} 
\right.
\end{equation*}
which implies that $\hat{\phi}=C$ (constant), for $t\in [0, T]$. Due to the appropriate choice of the state space $\mathcal{H}$, we have $\hat{\phi}=0$. By the constraint $\rho\int_{\Omega}\hat{\phi}\mathrm{d}x = \frac{1}{\kappa}\int_{\Gamma_1}\hat{\delta}\mathrm{d}\Gamma$, we get $\int_{\Gamma_1}\hat{\delta}\mathrm{d}\Gamma=0$.
This conclusion, combined with the boundary condition
\[
- \mathrm{div}_{\Gamma}\left(\sigma\nabla_{\Gamma} \hat{\delta}\right) + k \hat{\delta} = 0, \quad \textrm{on $\Gamma_1$}
\]
yields $\hat{\delta} = 0$.
We thus obtain $\{\hat{\Phi}(t), \hat{\Psi}(t)\}=0$, which contradicts~\eqref{step4_3}. Therefore, we complete the proof of inequality~\eqref{step4_1}.

Furthermore, the estimate~\eqref{step4_01} reduces to
\begin{equation}\label{step4_4}
\begin{split}
&(T-2\epsilon_0 - C(h, \rho, m, d)+\tau_2 T)\mathcal{E}(T) \\
&\leqslant C(T, \epsilon_0, \tau_i, h, m, \rho, d, k, N_1, N_2)\int_{0}^{T}\int_{\Gamma_1}\left[|\zeta \delta_{tt}|^2 +|\delta_t|^2 \right]\mathrm{d}\Gamma \mathrm{d}t.
\end{split}
\end{equation}
Notice that $0<\zeta(t)<1$, $t\in(0, T)$. Then by choosing $T$ large enough, the estimate~\eqref{step4_4} yields~\eqref{ineq:obs_abc}. This concludes the proof of Theorem~\ref{observability}.
\end{proof}

The estimate~\eqref{ineq:obs_abc} proposed in Theorem~\ref{observability} is the corresponding observability inequality of the problem~\eqref{1}-\eqref{4}. Indeed, setting $\mathcal{V} = \left[H^{1}(0, T; \mathcal{H})\right]$, $\mathcal{U} = \left[H^{1}(0, T; \mathcal{H})\right]^{\prime}$,  and $\mathcal{W}=L^2(0, T; \mathcal{H})$, we introduce a self-adjoint isomorphism $\Lambda: \mathcal{V}\rightarrow\mathcal{W}$ with
\begin{equation}\label{norm:spaces}
(\mathbf{ a}, \mathbf{ b})_{\mathcal{V}} = (\Lambda\mathbf{ a}, \Lambda\mathbf{ b} )_{\mathcal{W}}, \quad (\mathbf{ a}, \mathbf{ b})_{\mathcal{U}} = (\Lambda^{-1}\mathbf{ a}, \Lambda^{-1}\mathbf{ b} )_{\mathcal{W}}, \quad \|\Lambda\mathbf{ a}\|^2_{\mathcal{W}} = \|\mathbf{ a}\|^2_{\mathcal{V}} = \|\mathbf{ a}\|^2_{\mathcal{W}} + \|\frac{d}{dt}\mathbf{ a}\|^2_{\mathcal{W}}.
\end{equation}
Let the linear map $\mathcal{F}_{T}$ be as follows
\begin{equation}\label{linear_map}
	\mathcal{F}_{T}: \mathcal{U}\supset \mathrm{Dom}(\mathcal{F}_T) \rightarrow \mathcal{H}, \quad \mathcal{F}_{T}\mathbf{u}=\int_{0}^{T}S(T-t)\mathcal{B}\mathbf{u}\mathrm{d}t.
\end{equation}
Then we deduce
\[
(\mathcal{F}_{T}(\mathbf{u}), z)_{\mathcal{H}} = (\int_{0}^{T}S(T-t)\mathcal{B}\mathbf{u}\mathrm{d}t, z)_{\mathcal{H}} = (\mathbf{u}, \mathcal{B}^{\ast}S^{\ast}(T-t)z)_{\mathcal{W}}
\]
and
\begin{equation*}
	(\mathbf{u}, \mathcal{F}^{\ast}_{T}(z))_{\mathcal{U}} = (\Lambda^{-1}\mathbf{u}, \Lambda^{-1}\mathcal{F}^{\ast}_{T}(z) )_{\mathcal{W}} = (\mathbf{u}, \Lambda^{-2}\mathcal{F}^{\ast}_{T}(z) )_{\mathcal{W}}.
\end{equation*}
Thanks to the fact $(\mathcal{F}_{T}(\mathbf{u}), z)_{\mathcal{H}}=(\mathbf{u}, \mathcal{F}^{\ast}_{T}(z))_{\mathcal{U}}$, we obtain
\begin{equation}\label{map:rep}
\left(\Lambda^{-2}\mathcal{F}^{\ast}_{T}(z)\right)(t) = \mathcal{B}^{\ast}S^{\ast}(T-t)z.
\end{equation}
Using the solution $Z(t) = S^{\ast}(T-t)z$ of the backward system~\eqref{adjoint:abstract}, it implies from~\eqref{norm:spaces} and \eqref{map:rep} that
\begin{equation*}
\begin{split}
\|\mathcal{F}^{\ast}_{T}(z)\|^2_{\mathcal{U}} &= \|\Lambda^{-1}\mathcal{F}^{\ast}_{T}(z)\|^2_{\mathcal{W}}\\
& = \|\Lambda \mathcal{B}^{\ast}S^{\ast}(T-t)z\|^2_{\mathcal{W}}\\
& = \|\mathcal{B}^{\ast}Z(t) \|^2_{\mathcal{V}} \\
&= \frac{1}{m^2}\int_{0}^{T}\left\{ \|\delta_t\|^2_{L^2(\Gamma_1)}+\|\delta_{tt}\|^2_{L^2(\Gamma_1)}\right\}\mathrm{d}t.
\end{split}
\end{equation*}
Therefore, the abstract inequality~\eqref{ineq:obs} takes the equivalent form~\eqref{ineq:obs_abc}, from which one immediately derives the following result.

\begin{corollary}\label{control:exact}
The constrained potential problem~\eqref{1}-\eqref{4} is exactly controllable in time $T$.
\end{corollary}

\subsection{Exact Controllability}
\label{subsec:controllability}

As discussed above, we have provided an answer to question (Q1). 
In the sequel we shall deduce an explicit control function $f\in\mathcal{U}$ that can steer any initial state belonging to $\mathcal{H}$ of the acoustic motion to a predetermined state in time $T$. 

Let us first introduce the functional $\mathcal{J}:\mathcal{H}\rightarrow\mathbb{R}$ as
\begin{equation*}
\begin{split}
&\mathcal{J}(\phi_0, \delta_0, \phi_1, \delta_1) \\
&= \frac{1}{2}\int_{0}^{T}\left\{ \|\delta_t\|^2_{L^2(\Gamma_1)}+\|\zeta\delta_{tt}\|^2_{L^2(\Gamma_1)}\right\}\mathrm{d}t + \int_{\Omega}\rho u_{t}(0)\phi_t(0)\mathrm{d}x + \int_{\Gamma_1}m c^2 v_{t}(0)\delta_t(0)\mathrm{d}\Gamma\\
&\quad - \int_{\Omega}\rho u(0)\phi_{tt}(0)\mathrm{d}\Gamma - \int_{\Gamma_1}\rho c^2 v(0)\phi_t(0)\mathrm{d}\Gamma + \int_{\Gamma_1}\rho c^2 u(0)\delta_t(0)\mathrm{d}\Gamma\\
&\quad + \int_{\Gamma_1}c^2 d v(0)\delta_t(0)\mathrm{d}\Gamma - \int_{\Gamma_1}m c^2 v(0)\delta_{tt}(0)\mathrm{d}\Gamma,
\end{split}
\end{equation*}
where $({\phi}, {\delta})$ is the strong solution of system~\eqref{adjoint} subject to the initial data~$({\phi}_0, {\delta}_0, {\phi}_1, {\delta}_1)$, and $({u}, {w})$ is the strong solution of problem~\eqref{1}-\eqref{4}. Then we notice that
\begin{itemize}
\item the functional $\mathcal{J}$ is continuous and strongly convex;
\item and by the estimate~\eqref{step4_4} derived in the proof of Theorem~\ref{observability}, one has
\[
\mathcal{J}(\phi_0, \delta_0, \phi_1, \delta_1) \geqslant C_{T}\|(\phi_0, \delta_0, \phi_1, \delta_1)\|^2_{\mathcal{H}} - C_0\|(\phi_0, \delta_0, \phi_1, \delta_1)\|^2_{\mathcal{H}}
\]
which means $\mathcal{J}$ is coercive. 
\end{itemize}
Using the results in~\cite[Chapter 8.2.4, pp. 477-478]{evans2022partial}, the functional $\mathcal{J}$ possesses a unique minimizer, which solves the corresponding Euler--Lagrange equation
\begin{equation}\label{euler-lagrange}
\begin{split}
&\int_{0}^{T}\int_{\Gamma_1}\left(\hat{\delta}_t\delta_t+ \zeta^2\hat{\delta}_{tt}\delta_{tt} \right) \mathrm{d}\Gamma \mathrm{d}t + \int_{\Omega}\rho u_{t}(0)\phi_t(0)\mathrm{d}x + \int_{\Gamma_1}m c^2 v_{t}(0)\delta_t(0)\mathrm{d}\Gamma\\
&\quad - \int_{\Omega}\rho u(0)\phi_{tt}(0)\mathrm{d}\Gamma - \int_{\Gamma_1}\rho c^2 v(0)\phi_t(0)\mathrm{d}\Gamma + \int_{\Gamma_1}\rho c^2 u(0)\delta_t(0)\mathrm{d}\Gamma\\
&\quad + \int_{\Gamma_1}c^2 d v(0)\delta_t(0)\mathrm{d}\Gamma - \int_{\Gamma_1}m c^2 v(0)\delta_{tt}(0)\mathrm{d}\Gamma= 0
\end{split}
\end{equation}
for every strong solution $(\phi, \delta)$ of~\eqref{adjoint}.

Next, we proceed to state another main result.
\begin{theorem}\label{controllability}
Suppose that $(\hat{\phi}, \hat{\delta}, \hat{\phi}_t, \hat{\delta}_t)$ is the minimizer of the functional $\mathcal{J}(\cdot)$. Then the control function 
\[
f = \frac{1}{c^2}\left[\hat{\delta_t} - (\zeta^2\hat{\delta}_{tt})_t\right]
\]
ensures the exact controllability of problem~\eqref{1}-\eqref{4}.
\end{theorem}
\begin{proof}
Let $(u, v)$ and $(\phi, \delta)$ be the strong solutions of problem~\eqref{1}-\eqref{4} and~\eqref{adjoint}, respectively. Assume that $(\phi, \delta)$ is smooth enough. Then by integrating by parts, we deduce from the first equation of~\eqref{adjoint} that
\begin{equation*}
\begin{split}
0=&\int_{0}^{T}\int_{\Omega}\rho u\left( {\phi}_{ttt} - c^2\triangle{\phi_t}\right)\mathrm{d}x \mathrm{d}t\\
=&\left[\int_{\Omega}\rho u {\phi}_{tt} \mathrm{d}x\right]_{t=0}^{t=T} - \left[\int_{\Omega}\rho u_{t} {\phi_t}\mathrm{d}x \right]_{t=0}^{t=T} + \int_{0}^{T}\int_{\Omega}\rho c^2\left( \triangle u{\phi_t} - u\triangle{\phi_t}\right)\mathrm{d}x \mathrm{d}t\\
=&\left[\int_{\Omega}\rho u {\phi}_{tt} \mathrm{d}x\right]_{t=0}^{t=T} - \left[\int_{\Omega}\rho u_{t} {\phi_t}\mathrm{d}x \right]_{t=0}^{t=T} + \int_{0}^{T}\int_{\Gamma}\rho c^2 \left(\partial_{\mathbf{n}}u{\phi_t} - u\partial_{\mathbf{n}}{\phi_t} \right) \mathrm{d}\Gamma \mathrm{d}t,
\end{split}
\end{equation*}
which implies from the boundary conditions that
\begin{equation} \label{explicit_1}
\left[\int_{\Omega}\rho u {\phi}_{tt} \mathrm{d}x\right]_{t=0}^{t=T} - \left[\int_{\Omega}\rho u_{t} {\phi_t}\mathrm{d}x \right]_{t=0}^{t=T} + \int_{0}^{T}\int_{\Gamma_1}\rho c^2 \left(v_{t}{\phi_t} - u{\delta}_{tt} \right) \mathrm{d}\Gamma \mathrm{d}t=0.
\end{equation}
Similarly, we also deduce from the second equation of~\eqref{adjoint} and~\eqref{2} that
\begin{equation}\label{explicit_2}
\begin{split}
0=&\int_{0}^{T}\int_{\Gamma_1}c^{2}v\left( \rho {\phi}_{tt} + m {\delta}_{ttt} - \mathrm{div}_{\Gamma}\left(\sigma\nabla_{\Gamma} {\delta_t}\right) - d {\delta}_{tt} + k {\delta_t} \right)\mathrm{d}\Gamma \mathrm{d}t\\
=&\left[\int_{\Gamma_1}\rho c^{2}v  {\phi_t}\mathrm{d}\Gamma\right]_{t=0}^{t=T} - \int_{0}^{T}\int_{\Gamma_1}\rho c^{2}v_{t}  {\phi_t}\mathrm{d}\Gamma \mathrm{d}t + \int_{0}^{T}\int_{\Gamma_1}m c^{2}v  {\delta}_{ttt}\mathrm{d}\Gamma \mathrm{d}t \\
&+\int_{0}^{T}\int_{\Gamma_1}c^{2}v\left[- \mathrm{div}_{\Gamma}\left(\sigma\nabla_{\Gamma} {\delta_t}\right)  - d {\delta}_{tt} + k {\delta_t} \right]\mathrm{d}\Gamma \mathrm{d}t
\end{split}
\end{equation}
and
\begin{equation} \label{explicit_3}
\begin{split}
 &\int_{0}^{T}\int_{\Gamma_1}m c^{2}v {\delta}_{ttt}\mathrm{d}\Gamma \mathrm{d}t \\
 &= \left[\int_{\Gamma_1}m c^{2}v  {\delta}_{tt}\mathrm{d}\Gamma\right]_{t=0}^{t=T} - \left[\int_{\Gamma_1}m c^{2}v_{t}  {\delta_t}\mathrm{d}\Gamma\right]_{t=0}^{t=T} + \int_{0}^{T}\int_{\Gamma_1}m c^{2}v_{tt}  {\delta_t}\mathrm{d}\Gamma \mathrm{d}t\\
 &= \left[\int_{\Gamma_1}m c^{2}v  {\delta}_{tt}\mathrm{d}\Gamma\right]_{t=0}^{t=T} - \left[\int_{\Gamma_1}m c^{2}v_{t}  {\delta_t}\mathrm{d}\Gamma\right]_{t=0}^{t=T} \\
 &\quad+ \int_{0}^{T}\int_{\Gamma_1} c^{2}  {\delta_t}\left[ -\rho u_{t} + \mathrm{div}_{\Gamma}\left(\sigma\nabla_{\Gamma} {v}\right) -d v_{t} - k v - f\right]\mathrm{d}\Gamma \mathrm{d}t\\
  &=  \left[\int_{\Gamma_1}m c^{2}v  {\delta}_{tt}\mathrm{d}\Gamma\right]_{t=0}^{t=T} - \left[\int_{\Gamma_1}m c^{2}v_{t}  {\delta_t}\mathrm{d}\Gamma\right]_{t=0}^{t=T} \\
  &\quad -\left[\int_{\Gamma_1}\rho c^{2}  {\delta_t} u\mathrm{d}\Gamma\right]_{t=0}^{t=T} + \int_{0}^{T}\int_{\Gamma_1}\rho c^{2}  {\delta}_{tt} u\mathrm{d}\Gamma \mathrm{d}t - \int_{0}^{T}\int_{\Gamma_1}c^{2}\delta_t f\mathrm{d}\Gamma \mathrm{d}t\\
 &\quad -\left[\int_{\Gamma_1} c^{2} d {v}\delta_t\mathrm{d}\Gamma \right]_{t=0}^{t=T}+ \int_{0}^{T}\int_{\Gamma_1} c^{2}  {v}\left[ \mathrm{div}_{\Gamma}\left(\sigma\nabla_{\Gamma} {\delta_t}\right) + d \delta_{tt} - k \delta_t\right]\mathrm{d}\Gamma \mathrm{d}t.
 \end{split}
\end{equation}

Combing~\eqref{explicit_1}, \eqref{explicit_2} and~\eqref{explicit_3} yields
\begin{equation}\label{explicit_4}
\begin{split}
&\left[\int_{\Omega}\rho u {\phi}_{tt} \mathrm{d}x\right]_{t=0}^{t=T} - \left[\int_{\Omega}\rho u_{t} {\phi_t}\mathrm{d}x \right]_{t=0}^{t=T} + \left[\int_{\Gamma_1}\rho c^{2}v  {\phi_t}\mathrm{d}\Gamma\right]_{t=0}^{t=T}\\
&+\left[\int_{\Gamma_1}m c^{2}v  {\delta}_{tt}\mathrm{d}\Gamma\right]_{t=0}^{t=T} - \left[\int_{\Gamma_1}m c^{2}v_{t}  {\delta_t}\mathrm{d}\Gamma\right]_{t=0}^{t=T}-\left[\int_{\Gamma_1}\rho c^{2}  {\delta_t} u\mathrm{d}\Gamma\right]_{t=0}^{t=T}\\
&-\left[\int_{\Gamma_1} c^{2} d {v}\delta_t\mathrm{d}\Gamma \right]_{t=0}^{t=T} - \int_{0}^{T}\int_{\Gamma_1}c^{2}\delta_t f\mathrm{d}\Gamma \mathrm{d}t=0.
 \end{split}
\end{equation}
Plugging $f = -\frac{1}{c^2}\left(\hat{\delta_t} - (\zeta^2\hat{\delta}_{tt})_t\right)$ into~\eqref{explicit_4} and integrating by parts, we obtain
\begin{equation}\label{explicit_5}
\begin{split}
&\left[\int_{\Omega}\rho u {\phi}_{tt} \mathrm{d}x\right]_{t=0}^{t=T} - \left[\int_{\Omega}\rho u_{t} {\phi_t}\mathrm{d}x \right]_{t=0}^{t=T} + \left[\int_{\Gamma_1}\rho c^{2}v  {\phi_t}\mathrm{d}\Gamma\right]_{t=0}^{t=T}\\
&+\left[\int_{\Gamma_1}m c^{2}v  {\delta}_{tt}\mathrm{d}\Gamma\right]_{t=0}^{t=T} - \left[\int_{\Gamma_1}m c^{2}v_{t}  {\delta_t}\mathrm{d}\Gamma\right]_{t=0}^{t=T}-\left[\int_{\Gamma_1}\rho c^{2}  {\delta_t} u\mathrm{d}\Gamma\right]_{t=0}^{t=T}\\
&-\left[\int_{\Gamma_1} c^{2} d {v}\delta_t\mathrm{d}\Gamma \right]_{t=0}^{t=T}+ \int_{0}^{T}\int_{\Gamma_1} \left(\hat{\delta}_{t}\delta_t  + \zeta^2\hat{\delta}_{tt} \delta_{tt}\right)\mathrm{d}\Gamma \mathrm{d}t=0.
 \end{split}
\end{equation}
Subtracting~\eqref{explicit_5} from~\eqref{euler-lagrange} gives
\begin{equation}\label{explicit_6}
\begin{split}
&\int_{\Omega}\rho u_{t}(T)\phi_t(T)\mathrm{d}x + \int_{\Gamma_1}m c^2 v_{t}(T)\delta_t(T)\mathrm{d}\Gamma- \int_{\Omega}\rho u(T)\phi_{tt}(T)\mathrm{d}x - \int_{\Gamma_1}\rho c^2 v(T)\phi_t(T)\mathrm{d}\Gamma\\
&  + \int_{\Gamma_1}\rho c^2 u(T)\delta_t(T)\mathrm{d}\Gamma+ \int_{\Gamma_1}c^2 d v(T)\delta_t(T)\mathrm{d}\Gamma - \int_{\Gamma_1}m c^2 v(T)\delta_{tt}(T)\mathrm{d}\Gamma = 0.
\end{split}
\end{equation}
Choosing $\phi_1=0$ and $\delta_1=0$, one easily deduces 
\[
\int_{\Omega}\rho u(T)\phi_{tt}(T)\mathrm{d}x + \int_{\Gamma_1}m c^2 v(T)\delta_{tt}(T)\mathrm{d}\Gamma = 0, \quad \forall (\phi_0, \delta_0),
\]
which implies $u(T) = 0$ and $v(T) = 0$. Furthermore, the equation \eqref{explicit_6} reduces to
\[
\int_{\Omega}\rho u_{t}(T)\phi_t(T)\mathrm{d}x + \int_{\Gamma_1}m c^2 v_{t}(T)\delta_t(T)\mathrm{d}\Gamma = 0, \quad \forall (\phi_0, \delta_0).
\]
Hence, $(u_{t}(T), v_{t}(T))=(0, 0)$. This concludes the proof of Theorem~\ref{controllability}.
\end{proof}

\section{Strong Mixing}
\label{sec:mixing}

The goal of this section is to provide an answer to (Q2) for the problem~\eqref{1}-\eqref{4} perturbed by some random force $f$ of white noise type.

%\subsection{Stochastic evolution equation}

In the following, $(\Omega, \mathcal{F}, \mathbb{P})$ is a probability space with a right-continuous increasing family $\mathcal{F}=(\mathcal{F}_t)_{t\geqslant 0}$ of sub-$\sigma$-fields of $\mathcal{F}_0$ each containing $\mathbb{P}$-null sets. $\mathbb{E}(\cdot)$ stands for expectation with respect to the probability measure $\mathbb{P}$.  
Let $W(t)$, $t\geq 0$ be a $\mathcal{U}$-valued $Q$-Wiener process on~$(\Omega, \mathcal{F}, \mathbb{P})$, with the covariance operator $Q$. Here, $Q$ is a symmetric negative operator, which is also assumed to be linear and bounded on $\mathcal{U}$. Furthermore, the trace of $Q$, denoted by $\mathrm{Tr}Q$, is finite.

Since the random force $f: \mathbb{R}^{+}\rightarrow\mathcal{U}$ is of white noise type, the problem~\eqref{1}-\eqref{4} can be reformulated to be a linear stochastic evolution equation driven by a $Q$-Wiener process as follows 
\begin{equation} \label{linearSDE}
\left
    \{
        \begin{array}{ll}
            \mathrm{d}\mathbf{X}(t)=\mathcal{A}\mathbf{X}(t)\mathrm{d}t+ \mathcal{B}\mathrm{d}W(t),\\
            \mathbf{X}(0)=\mathbf{x}.  
        \end{array}
\right.
\end{equation}
By Proposition~\ref{prop:operator}, a linear operator $\mathcal{Q}_t$, $t > 0$, defined by
\[
	\mathcal{Q}_{t}\mathbf{x}=\int_{0}^{t}S(s)\mathcal{B}\mathcal{B}^{\ast}S^{\ast}(s)\mathbf{x} \mathrm{d}s, \quad \mathbf{x}\in\mathcal{H}
\]
is of trace class .
Then using Theorem 5.3.1 in \cite{PratoZabczyk}, there exists a unique $\mathcal{H}$-valued $\mathcal{F}_{t}$-adapted process $\mathbf{X}(t, \mathbf{x})$, $t\geq 0$, satisfies the following integral equation  
\begin{eqnarray*}
	\mathbf{X}(t, \mathbf{x}) =S(t)\mathbf{x} +\int_{0}^{t}S(t-s)\mathcal{B}\mathrm{d}W(s), \quad t\in [0, T]
\end{eqnarray*}
for any $\mathbf{x} \in \mathcal{H}$ is $\mathcal{F}_0$-measurable and $\mathbb{E}\|\mathbf{x}\|_{\mathcal{H}}^2\leq \infty$, which means $\mathbf{X}(t, \mathbf{x})$ is a mild solution of equation~\eqref{linearSDE}.

Let $P_t$ stands for the transition function of the family defined as the law of the solution $X(t, \mathbf{x})$ under the probability measure $\mathbb{P}$
\[
	P_t(\mathbf{x}, \Xi)=\mathbb{P}(X(t, \mathbf{x})\in \Xi), \quad \Xi\in B(\mathcal{H}), \quad t\geqslant 0.	
\]
The corresponding Markov semigroups are given by
\begin{equation*}
 \begin{aligned}
 	M_{t}&:  L^{\infty}(\mathcal{H})\rightarrow L^{\infty}(\mathcal{H}), \quad M_{t}f(\mathbf{x})=\int_{\mathcal{H}}P_{t}(\mathbf{x}, \mathrm{d}z)f(z),\\
	M^{\ast}_{t}&:  \mathcal{P}(\mathcal{H})\rightarrow \mathcal{P}(\mathcal{H}), \quad M^{\ast}_{t}\lambda(\Xi)=\int_{\mathcal{H}}P_{t}(\mathbf{x}, \Xi)\lambda(\mathrm{d}\mathbf{x}).
\end{aligned}
\end{equation*}
Note that these two semigroups satisfy the duality relation 
\[
(M_{t}f, \lambda)=(f, M_{t}^{\ast}\lambda), \quad \forall f\in C(\mathcal{H}), \quad \forall\lambda\in \mathcal{P}(\mathcal{H}). 
\]
A probability measure $\mu\in\mathcal{P}(\mathcal{H})$ is said to be invariant with respect to $M_{t}$ if and only if $M^{\ast}_{t}\mu = \mu$ for each $t\geqslant0$. We call a Markov semigroup $M_{t}$ is $t_0$-regular if all transition probabilities $P_{t_0}(\mathbf{x}, \cdot)$, $\mathbf{x}\in\mathcal{H}$, are mutually equivalent. 

In the sequel, we give a sufficient and necessary condition for the regularity property.

\begin{lemma}\label{regularity}
The Markov semigroup $M_{t}$, $t\geqslant0$, is $t$-regular if and only if the observability inequality~\eqref{ineq:obs_abc} holds.
\end{lemma}
\begin{proof}
Since $X(t, \mathbf{x})$ is a mild solution of the linear equation~\eqref{linearSDE}, it is obvious that $X(t, \mathbf{x})$ is a Gaussian random variable with mean $S(t)\mathbf{x}$ and covariance $\mathcal{Q}_{t}$. If we assume that the range of $S(t)$ is a subset of the range of $\mathcal{Q}_{t}^{\frac{1}{2}}$, that is,  
$\textrm{Ran} S(t)\subset \textrm{Ran}\mathcal{Q}_{t}^{\frac{1}{2}}$, $t>0$,
then we could define a linear bounded operator $R(t)=\mathcal{Q}_{t}^{-\frac{1}{2}}\circ S(t)$.  From the Cameron-Martin formula~\cite[Proposition 2.24, pp. 68]{da2014stochastic}, it implies that for any $\mathbf{x}$, $\mathbf{y}\in \mathcal{H}$
 \begin{equation}\label{CMformula}
 \begin{aligned}
	P_t(\mathbf{x}, \Xi) &= \int_{\Lambda}G(t,\mathbf{x}, \tilde{\mathbf{x}})P_t(0, \mathrm{d}\tilde{\mathbf{x}})\\
	&= \int_{\Lambda}\frac{G(t,\mathbf{x}, \tilde{\mathbf{x}})}{G(t,\mathbf{y}, \tilde{\mathbf{x}})}G(t,\mathbf{y}, \tilde{\mathbf{x}})P_t(0, \mathrm{d}\tilde{\mathbf{x}}) = \int_{\Lambda}\frac{G(t,\mathbf{x}, \tilde{\mathbf{x}})}{G(t,\mathbf{y}, \tilde{\mathbf{x}})}P_t(\mathbf{y}, \mathrm{d}\tilde{\mathbf{x}})
 \end{aligned}
\end{equation} 
with
\[
G(t, \cdot, \tilde{\mathbf{x}})=\exp\left\{\big(R(t)\cdot, \mathcal{Q}_{t}^{-\frac{1}{2}}\tilde{\mathbf{x}}\big)_{\mathcal{H}}-\frac{1}{2}\|R(t)\cdot\|_{\mathcal{H}}^2\right\}.
\]	
Using~\eqref{CMformula} we can deduce that all transition probabilities $P_{t}(\mathbf{x}, \cdot)$, $\mathbf{x}\in\mathcal{H}$, are mutually equivalent, which means $M_{t}$, $t\geqslant0$, is $t$-regular.
This concludes a sufficient and necessary condition for the regularity of $M_{t}$ is that $\textrm{Ran} S(t)\subset \textrm{Ran}\mathcal{Q}_{t}^{\frac{1}{2}}$.

Next, we shall show that the condition $\textrm{Ran} S(t)\subset \textrm{Ran}\mathcal{Q}_{t}^{\frac{1}{2}}$ is equivalent to the fact that the equation~\eqref{evolution} is null controllable in time $T$. 
Indeed, we note that
\begin{equation*}
\begin{split}
\|\mathcal{Q}_{T}^{\frac{1}{2}}z\|^2 = (\mathcal{Q}_{T}z, z)_{\mathcal{H}} &= (\int_{0}^{T}S(s)\mathcal{B}\mathcal{B}^{\ast}S^{\ast}(s)z \mathrm{d}s, z)_{\mathcal{H}} \\
&= (\int_{0}^{T}S(T-s)\mathcal{B}\mathcal{B}^{\ast}S^{\ast}(T-s)z \mathrm{d}s, z)_{\mathcal{H}}=\|\mathcal{B}^{\ast}S^{\ast}(T-s)z\|_{\mathcal{W}}
\end{split}
\end{equation*}
and the linear map~\eqref{linear_map} satisfies
\[
 (F_{T}F^{\ast}_{T}z, z)_{\mathcal{H}}=\| F^{\ast}_{T}z\|^2_{_{\mathcal{U}}}=\| \Lambda^{-1}F^{\ast}_{T}z\|^2_{_{\mathcal{W}}} = \|\Lambda\mathcal{B}^{\ast}S^{\ast}(T-t)z\|_{\mathcal{W}}.
\]
Then by the definition of $\mathcal{Q}_{T}^{\frac{1}{2}}$ and the isomorphism $\Lambda$, we see $\textrm{Ran}\mathcal{Q}_{T}^{\frac{1}{2}} =\textrm{Ran}F_T$.
In addition, due to the representation $\mathbf{X}(T) =S(T)\mathbf{x} +F_T\mathbf{u}$, we know $\textrm{Ran}F_T$ consists of all states reachable in time $T$ from the zero state.
Moreover equation~\eqref{evolution} is null controllable in time $T$ if and only if $\textrm{Ran}S(T) \subset  \textrm{Ran}F_T = \textrm{Ran}\mathcal{Q}_{T}^{\frac{1}{2}} $. As discussed in Subsection~\ref{subsec:obs}, the equivalence between null controllability and observability is established. This completes the proof of Lemma~\ref{regularity}.
\end{proof}

\begin{theorem}\label{mixing}
Suppose the boundary of the domain satisfies the geometric condition~\eqref{G1}. Then the problem~\eqref{1}-\eqref{4}, subject to a random force of white noise type, has a unique invariant measure $\mu$ satisfying
\[
	\lim_{t\rightarrow\infty}\mathbb{E}\ell(X(t, \mathbf{x}))=\int_{\mathcal{H}}\ell(z)\mu(\mathrm{d}z)
\]
for any $\ell\in C_{b}(\mathcal{H})$, 
which means the problem~\eqref{1}-\eqref{4} is strongly mixing.
\end{theorem}
\begin{proof}
Thanks to Young's inequality, we have
 \begin{equation}\label{estimate1}
 \begin{aligned}
 	\|X(t, \mathbf{x})\|^{2}_{\mathcal{H}}&=\|\mathbf{x}\|^{2}_{\mathcal{H}}-{\frac{2d}{\rho}}\int_{0}^{t}\int_{\Gamma_1}\delta^{2}_{t}\mathrm{d}\Gamma_1 \mathrm{d}t-{\frac{2}{\rho}}\int_{\Gamma_1}\int_{0}^{t}\delta_{t}\mathrm{d}W(s) \mathrm{d}\Gamma_1 \\
	&\leq \|\mathbf{x}\|^{2}_{\mathcal{H}}-{\frac{2d}{\rho}}\int_{0}^{t}\int_{\Gamma_1}\delta^{2}_{t}\mathrm{d}\Gamma_1 \mathrm{d}t+{\frac{2d}{\rho}}\int_{\Gamma_1}\big(\int_{0}^{t}\delta_{t}\mathrm{d}W(s)\big)^2\mathrm{ d}\Gamma_1+\frac{|\Gamma_1|}{4d{\rho}}.
\end{aligned}
\end{equation} 
By the It\^o isometry $\int_{\Gamma_1}\mathbb{E}\big(\int_{0}^{t}\delta_{t}\mathrm{d}W(s)\big)^2 \mathrm{d}\Gamma_1=\int_{\Gamma_1}\mathbb{E}\big(\int_{0}^{t}\delta^2_{t}\mathrm{d}s\big) \mathrm{d}\Gamma_1$,  we deduce from (\ref{estimate1}) that
 \begin{equation}\label{estimate2}
 	\mathbb{E}\|X(t, \mathbf{x})\|^{2}_{\mathcal{H}}\leq\mathbb{E}\|\mathbf{x}\|^{2}_{\mathcal{H}}+\frac{|\Gamma_1|}{4d{\rho}}.
\end{equation}
Due to the relation 
\[
\sup\limits_{t>0}\mathbb{E}\|X(t, 0)\|^{2}_{\mathcal{H}}=\sup\limits_{t\geq0}\textrm{Tr}Q_t,
\] 
the estimate~\eqref{estimate2} yields $\sup_{t\geq0}\textrm{Tr}Q_t < \infty$. 
Furthermore, there exists an invariant measure, denoted by $\mu$, for the problem~\eqref{1}-\eqref{4}, which is derived from Theorem 6.2.1 in~\cite{PratoZabczyk}. 

Based on Lemma \ref{regularity}, we obtain the corresponding Markov semigroup $M_t$, $t\geqslant 0$, is $t_0$-regular for some $t_0>0$. Then from Doob's theorem~\cite{Doob} it implies that $\mu$ is the unique invariant measure for the semigroup $M_t$, which is also strongly mixing.
\end{proof}

\section{Conclusion}
\label{sec:conclusion}

In this paper, we have continued our investigation into the controllability and mixing properties of acoustic wave motion driven by a boundary random force. Building on our previous work~\cite{jiao2025mixing}, which focused on the wave equation with acoustic boundary conditions on locally reacting surfaces~($\sigma = 0$), we have extended our study to the more complex scenario of non-locally reacting surfaces~($\sigma \neq 0$). By employing the multiplier method, we have established the observability inequality, a crucial step for proving the controllability and mixing properties of our wave system. We have successfully addressed the significant mathematical challenge associated with estimating the integral~$\int_{0}^{T}\int_{\Gamma_1}\zeta^2|\phi_t|^2 \mathrm{d}\Gamma \mathrm{d}t$. Moreover, we have constructed an explicit control that ensures the null controllability of our wave system via a minimizer of a suitable functional.
Our results contribute to the understanding of controllability and mixing properties for the mathematical model discribing the interaction between acoustic wave motions and non-locally reacting surfaces.

\section*{Acknowledgement}
The authors thank the anonymous referee very much for the helpful suggestions.

\section*{Declaration of competing interest}
The authors have no conflicts to disclose.

\section*{References}

\bibliography{mybibfile}

\end{document}